\documentclass[onefignum,onetabnum,reqno,a4paper,final]{siamart250211}

\usepackage{lipsum}
\usepackage{amsmath,amssymb,amsfonts,latexsym,stmaryrd}
\usepackage{graphicx,float}
\usepackage{mathrsfs} 
\usepackage{color}
\usepackage{epstopdf}
\usepackage{algorithmic}
\usepackage{multirow}
\ifpdf
\DeclareGraphicsExtensions{.eps,.pdf,.png,.jpg}
\else
\DeclareGraphicsExtensions{.eps}
\fi

\def\O{\Omega}

\newcommand{\jump}[1]{\llbracket #1 \rrbracket}
\newcommand{\mean}[1]{\{\!\!\{#1\}\!\!\}}

\usepackage{booktabs}
\usepackage{float}

\newcommand\bu{\boldsymbol{u}}
\newcommand\bv{\boldsymbol{v}}
\newcommand\bx{\boldsymbol{x}}

\newcommand\bn{\boldsymbol{n}}

\newcommand\bw{\boldsymbol{w}}

\newcommand\bL{\mathbf{L}}
\newcommand\bM{\mathbf{M}}

\newcommand\bU{\boldsymbol{U}}

\newcommand\CE{\mathcal{E}}

\def\CT{{\mathcal T}}

\newcommand{\bt}{\boldsymbol{t}}

\newcommand\bsig{\boldsymbol{\sigma}}
\newcommand\btau{\boldsymbol{\tau}}

\newcommand\R{\mathbb{R}}

\renewcommand\H{\mathrm{H}}
\renewcommand\L{\mathrm{L}}

\newcommand\bH{\boldsymbol{\mathrm{H}}}

\newcommand\bI{\boldsymbol{\mathrm{I}}}
\newcommand\bX{\boldsymbol{\mathrm{X}}}
\newcommand\bzero{\boldsymbol{0}}

\renewcommand\O{\Omega}
\newcommand\DO{\partial\O}
\newcommand\Gm{\Gamma_{\mathrm{m}}}
\newcommand\Gi{\Gamma_{\mathrm{in}}}
\newcommand\Go{\Gamma_{\mathrm{out}}}
\newcommand\Gw{\Gamma_{\mathrm{wall}}}
\renewcommand\div{\mathop{\mathrm{div}}\nolimits}

\newcommand\HsO{\H^s(\O)}

\newcommand\buin{\boldsymbol{u}_{\mathrm{in}}}

\newcommand{\vertiii}[1]{{\left\vert\kern-0.25ex\left\vert\kern-0.25ex\left\vert #1 
		\right\vert\kern-0.25ex\right\vert\kern-0.25ex\right\vert}}

\newsiamremark{remark}{Remark}
\newsiamremark{hypothesis}{Hypothesis}
\crefname{hypothesis}{Hypothesis}{Hypotheses}
\newsiamthm{claim}{Claim}

\headers{div-conforming methods for flow-transport equations}{Khan, Mora, Ruiz-Baier, Vellojin}

\title{Divergence-conforming finite element methods for flow-transport coupling with osmotic effects\thanks{Updated: \today.
		\funding{A. Khan was supported by the  Sponsored Research \& Industrial Consultancy (SRIC), Indian Institute of Technology Roorkee, India through the faculty initiation grant MTD/FIG/100878, and  by SERB Core research grant CRG/2021/002569. D.  Mora was  supported by the National Agency for Research and Development, ANID-Chile through FONDECYT project 1220881, project \textsc{Anillo of Computational Mathematics for Desalination Processes} ACT210087, and by  Centro de Modelamiento Matemático (CMM), FB210005, BASAL funds for centers of excellence. R. Ruiz-Baier was  supported by the Australian Research Council through the \textsc{Future Fellowship} grant FT220100496 and \textsc{Discovery Project} grant DP22010316. J. Vellojin was  supported by the National Agency for Research and Development, ANID-Chile through FONDECYT Postdoctorado project 3230302. 
}}}

\author{Arbaz Khan\thanks{Department of Mathematics, Indian Institute of Technology Roorkee, Roorkee 247667, India. \email{arbaz@ma.iitr.ac.in}.}
	\and David Mora\thanks{GIMNAP-Departamento de Matem\'atica, Universidad del B\'io-B\'io, Casilla 5-C, Concepci\'on, Chile, and
		and C$\textrm{I}^2$MA, Universidad de Concepci\'on, Chile. \email{dmora@ubiobio.cl}.}
	\and Ricardo Ruiz-Baier\thanks{School of Mathematics,
		Monash University, 9 Rainforest Walk, Melbourne, Victoria 3800, Australia; and Universidad Adventista de Chile, Casilla 7--D Chillan, Chile. \email{ricardo.ruizbaier@monash.edu}.}
	\and Jesus Vellojin\thanks{Corresponding author. Departamento de Ciencias, Universidad Técnica Federico Santa María,Valparaíso, Chile. \email{jesus.vellojinm@usm.cl}.}}

\usepackage{amsopn}

\ifpdf
\hypersetup{
	pdftitle={Div-conforming methods for flow-transport equations},
	pdfauthor={Khan, Mora, Ruiz-Baier, Vellojin}
}
\fi

\definecolor{lightblue}{rgb}{0.22,0.45,0.70}

\newcommand{\triplenorm}[1]{\lvert\!\lvert\!\lvert #1
  \rvert\!\rvert\!\rvert}

\usepackage[normalem]{ulem}

\usepackage[mathlines]{lineno}
\numberwithin{figure}{section}
\numberwithin{table}{section}
\begin{document}
	
	\maketitle
	
	\begin{abstract}
		We propose a new model for the coupling of flow and transport equations with porous membrane-type conditions on part of the boundary. The governing equations consist of the Navier--Stokes equations coupled with an advection-diffusion equation, and we employ a Lagrange multiplier to handle the coupling between penetration velocity and transport on the membrane, while mixed boundary conditions are considered in the remainder of the boundary. We show existence and uniqueness of the continuous problem using a fixed-point argument. Next, an H(div)-conforming finite element formulation is proposed, and we address its a priori error analysis. The method uses an upwind approach that provides stability in the convection-dominated regime. We showcase a set of numerical examples validating the theory and illustrating the use of the new methods in the simulation of reverse osmosis processes. 
	\end{abstract}
	\begin{keywords}
		Flow-transport equations; Lagrange multipliers; Reverse osmosis; Divergence-conforming FEM.
	\end{keywords}
	\begin{AMS}
		65N30, 76D07, 76D05.
	\end{AMS}
	
	%%%%%%%%%%%%%%%%%%%%%%%%%%%%%%%%%%%%%%%%%%%%%%%%%%%%%%%%%%%%%%%%%%%%% 
	\section{Introduction}
%\subsection{Scope}
The coupling of Navier--Stokes and advection-diffusion equations is central in many applications in industry and engineering. One of such instances is the simulation of filtration occurring in water purification, where high velocity flow  with relatively high concentration of salt goes through a unit under reverse osmosis. Even without the mechanisms of ion transport and molecule interactions through the membrane, the fluid dynamics process already poses interesting challenges. In this simplified context, we can consider  water and salt transport coupled through both advection and  boundary interaction at the membrane -- stating that velocity is proportional to %a linear function of 
	the salt concentration. A similar model has been studied numerically in \cite{carro2022finite}, where the coupling on the membrane follows Nitsche's approach. Here we rewrite that model using a membrane Lagrange multiplier and provide a complete well-posedness analysis as well as the analysis of a divergence-conforming discretisation. 
	
%In analysing the coupled system and 
Separating for a moment the salt transport from the incompressible flow equations, we end up with a generalisation of the Navier--Stokes equations with slip boundary conditions which have been studied extensively starting from the conforming finite element methods (FEMs) proposed in \cite{verfurth86,verfurth91}. The off-diagonal bilinear form is different from the usual one in that it also includes a pairing between the normal trace of velocity and the Lagrange multiplier taking into account the membrane coupling. Another difficulty in the model considered herein is that a non-homogeneous boundary condition is needed at the inlet. Often the analysis is simply restricted to the case of homogeneous essential boundary conditions, but here the non-homogeneity is important as it permits that the coupling occurs (otherwise the membrane coupling vanishes, and the only solution is the trivial one). 
	We note that for smooth domains it is expected that discretisations are prone to the so-called Babu\v{s}ka paradox -- a variational crime associated with the approximation of the boundary, and where a sub-optimal convergence is expected (see more details in, e.g., \cite{kashiwabara2016penalty,urquiza2014weak}). Similar works focusing on slip boundary conditions imposed with Lagrange multipliers or with penalty can be found in \cite{liakos2001discretization,layton1999weak,zhou2016penalty}. In our case we restrict to polygonal boundaries, which are typically  encountered in the driving application of water desalination. 
	
We also stress that the coupling with salt transport adds complexity to the model. The unique solvability of the coupled problem is analysed by casting it as a fixed-point equation and using Banach's fixed-point theorem. The fixed-point operator consists of the solution operator of the Navier--Stokes equations with membrane (or mixed slip-type) boundary conditions, composed with the solution operator associated with an advection-diffusion equation. The unique solvability of the first sub-problem is established   using a Stokes linearisation combined with the fixed-point theory in the case of non-homogeneous boundary conditions. The unique solvability of the outer fixed-point problem results as a consequence of an assumption of smallness of data, which in our context reduces to impose a condition on the inlet velocity and on the  constitutive equation for the membrane interaction term. Some parts of these proofs are standard but as we have not found them in the precise context we need, they have been included in the paper for sake of completeness. 

Related works where non-homogeneous mixed boundary conditions are of importance include, for instance, the solvability of Navier--Stokes equations with free boundary \cite{le1993steady}, Boussinesq equations with leaking boundary conditions and Tresca slip \cite{kim2022steady}, the fixed-point analysis for Navier--Stokes equations with mixed boundary conditions \cite{manzoni2019saddle},   the analysis on viscous flows around obstacles with non-homogeneous boundary data \cite{ingram2011finite}, the solvability of Boussinesq equations with mixed (and non-homogeneous) boundary data \cite{arndt2020existence,ceretani2019boussinesq}, and the regularity of split between normal and tangential parts of the velocity as boundary conditions for the Navier--Stokes equations in \cite{ebmeyer2001steady}. 
	However it is important to mention that, up to our knowledge, the set of equations we face here has not been analysed in the existing literature.

The divergence-conforming discontinuous Galerkin (DG) method (introduced in \cite{cockburn2007note}) represents a very useful numerical approach that produces divergence-free velocities.  Additionally, velocity error estimates could be determined in a manner that is resilient to variations in pressure. Moreover, locally, conservation guarantees a divergence-form representation for the coupled systems at the discrete level.  Studies on this regard can be found in, for instance, \cite{badia24,buerger2019h,schroeder2017stabilised,john2017divergence}. 
     
	\paragraph{Plan}
	The rest of the paper is organised as follows. The remainder of this section lists useful notation to be used throughout the paper. Section~\ref{sec:model} is devoted to  the governing equations and the specific boundary conditions needed in a typical operation of a desalination unit. There we also derive a weak formulation and provide preliminary properties of the weak forms. In Section~\ref{sec:wellp} we conduct the analysis of existence and uniqueness of weak solution to the coupled system. We state an abstract result and show that the Navier--Stokes equations with membrane boundary conditions adhere to that setting. This section also describes the fixed-point analysis. Section~\ref{sec:fe} contains the definition of a conforming discretisation, a stabilisation technique proposed in \cite{verfurth91}, and then we define a new H(div)-conforming method and state main properties of the modified discrete variational forms. {This is accompanied by a brief discussion on a least-squares stabilised method using Lagrange multipliers, similar to that of \cite{verfurth86,verfurth91}}. The unique solvability of the discrete {divergence-free} problem is studied in Section~\ref{sec:discrete-wellp}, while the derivation of a priori error estimates is presented in Section~\ref{sec:cv}. Section \ref{sec:lagrange-multiplier} provides a brief description of a least-squares stabilised scheme using continuous elements. Qualitative properties of the proposed formulations are explored in Section~\ref{sec:experiments}, and we also confirm numerically the convergence rates predicted by the theory.

	\paragraph{Preliminaries and notation}
	Let $\O$ be a polygonal  bounded domain of $\R^2$ with Lipschitz boundary $\DO$.  
	   We employ standard notation for Lebesgue spaces, Sobolev spaces and their respective norms. Given $s\geq 0$ and $p\in[1,\infty]$, we denote by $\L^p(\O)$ Lebesgue space endowed with the norm $\Vert\cdot\Vert_{\L^p(\O)}$, while $\HsO$ denotes a Hilbert space. Vectors spaces and vector-valued functions are written in bold letters. For instance, for $s\geq 0$, we simply write $\bH^s(\O)$ instead of $[\H^s(\O)]^2$. If $s=0$, we use the convention $\H^0:=\L^2(\O)$ and $\bH^0:=\mathbf{L}^2(\O)$. For the sake of simplicity, the seminorms and norms in Hilbert spaces are denoted by $\vert \cdot\vert_{s,\Omega}$ and $\Vert\cdot\Vert_{s,\O}$, respectively. The unit outward normal at $\partial\Omega$ is denoted by $\bn:=(n_1,n_2)$, whereas $\boldsymbol{t}$  denotes the corresponding unit tangential vector 
    perpendicular to $\bn$ on $\partial\Omega$. Also, $\mathbf{0}$ denotes a generic null vector.	Let us also define for $s\geq 0$ the Hilbert space $\H(\div,\O):=\{\bv\in\mathbf{L}^2(\O):\div\bv\in\L^2(\O)\}$ whose norm is given by $\Vert\bv\Vert_{\div,\O}^2:=\Vert\bv\Vert_{0,\O}^2+\Vert\div\bv\Vert_{0,\O}^2$.
 We also recall that for a Hilbert space $H$ with inner product $(\cdot,\cdot)_H$,  $\mathcal{R}_H$ denotes the Riesz operator $H\rightarrow H'$  that to each $z$ associates the functional $f_z=\mathcal{R}_H z\in H'$ defined as
         $$
         \langle f_z,v\rangle_{H'\times H}=(z,v)_H\qquad \forall v \in H,
         $$
         with $\langle\cdot,\cdot\rangle_{H'\times H}$ denoting the duality pairing between $H'$ and $H$.
    $\mathcal{R}_H$ is one-to-one, and $\Vert \mathcal{R}_H\Vert_{\mathcal{L}(H,H')}=\Vert \mathcal{R}_H^{-1}\Vert_{\mathcal{L}(H',H)}=1$. Moreover, if $H$ is identified with $(H')'$, then $\mathcal{R}_H^{-1}=\mathcal{R}_{H'}$. Throughout the rest of the paper we abridge into $X\lesssim Y$ the inequality $X\leq CY$ with positive constant $C>0$ independent of $h$. Similarly for $X\gtrsim Y$. 

 %%%%%%%%%%%%%%%%%%%%%%%%%%%%%%%%%%%%%%%%%%%%%%%%%%%%%%%
	\section{Model problem}\label{sec:model}
    Let us consider that the boundary of $\Omega$  is decomposed as $\partial\O=\Gi\cup\Go\cup\Gw\cup\Gm$. The sub-boundary $\Gi$ corresponds to the inflow, $\Go$ is the outflow boundary, $\Gw$ is a no-slip no penetration boundary, and $\Gm$ is the porous membrane boundary.
	The fluid inside the channel is assumed to be Newtonian, incompressible, and composed only by water and salt. The density and viscosity are constant and positive $\rho_0, \mu_0>0$. The solute diffusivity through the solvent is given by $D_{0}$. It is also assumed that the effect of pressure drop inside the channel due to viscous effects on the permeate flux (solution-diffusion equation) is negligible.

	%%%%% Graph model %%%%%%%%%%%%%%
	\begin{figure}[t!]
		\centering
		\includegraphics[scale=0.675]{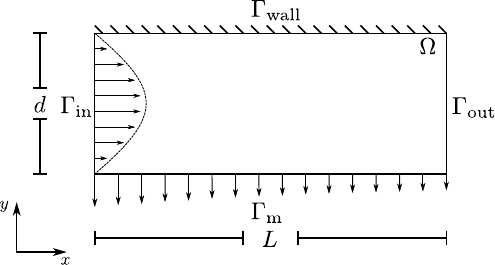}
		\caption{Cross-flow membrane filtration model with $\Gamma_i$ defined such that there is no corner points on $\Gm\cap\Gi$.}
		\label{fig:dibujo_canal}
	\end{figure}
	%%%%%%%%%%%%%%% 
	
	The resulting model is a coupling between Navier--Stokes and convection--diffusion equations:
	\begin{linenomath}
		\begin{align}
			\label{eq:coupled-model}
			-\mu_0\Delta\bu+\rho_0\bu\cdot\nabla\bu+\nabla p&= \bzero  \quad \text{ in }\Omega,\nonumber\\ 
			\nabla\cdot\bu&=0 \quad \text{ in }\Omega,\\ 
			-D_0\Delta \theta+\bu\cdot\nabla \theta&=0 \quad \text{ in }\Omega. \nonumber    \end{align}\end{linenomath}    
	Here, $\bu(\bx)$, $p(\bx)$, and $\theta(\bx)$ represent the fluid velocity, pressure and concentration profile, respectively. Along with \eqref{eq:coupled-model} we have the following set of boundary conditions:
	\begin{linenomath}
		\begin{subequations}\begin{align}	
				\bu&=\buin,\;\theta=\theta_{\mathrm{in}}&\text{ on }\Gi,\label{model-eq4}\\
				%\mu_0(\nabla\bu+(\nabla\bu)^t)\bn&=\boldsymbol{0},\; p=0,&\text{ on }\Gamma_{N},\\
				-D_0\nabla \theta\cdot\bn&=0&\text{ on }\Go,\label{model-eq5}\\
				\mu_0(\nabla\bu)\bn-p\bn&=\boldsymbol{0}&\text{ on }\Go,\label{model-eq6}\\
				\bu&=\boldsymbol{0}&\text{ on }\Gw,\label{model-eq6*}\\
				\left(\theta\bu-D_0\nabla \theta\right)\cdot \bn&=0&\text{ on }\Gm\cup\Gw, \label{model-eq7}\\
				\bu\cdot \boldsymbol{t}&=0&\text{ on }\Gm,\label{model-eq8}\\
				\bu\cdot\bn&=g(\theta)&\text{ on }\Gm.\label{model-eq9}
	\end{align}\end{subequations}\end{linenomath}
	In \eqref{model-eq4}, a parabolic inflow is considered with a fixed salt concentration representing a fully developed sea water flow trough the channel.     In \eqref{model-eq5}--\eqref{model-eq6} we have a zero salt flux and \textit{do nothing} boundary conditions. A zero Dirichlet boundary condition is imposed for the velocity across the impermeable wall. A full salt rejection is considered in the walls and the membrane, which is represented by \eqref{model-eq7}. In \eqref{model-eq9} we have the permeability condition, where the quantity $g(\theta)$ denotes the flow velocity at the membrane as a function of the concentration and it can be represented using the Darcy--Starling law. In order to enforce compatibility of the inlet and membrane boundary conditions we make precise the form of the parabolic profile as follows (assuming that the bottom left corner of the domain is located at the origin)
    \[ \bu_{\mathrm{in}} = \frac{1}{d}\bigl(u_0[d-y]y,-g(\theta_{\mathrm{in}})[d-y]\bigr)^{\tt t}, \quad \bu|_{\Gm} = (0,-g(\theta))^{\tt t},\]
    and we note that both $x$ and $y$ components of the velocity datum are continuous at the corner point $\Gi\cap \Gm$. This, or actually any smooth vertical inlet velocity that goes from  $u_\mathrm{in,2}=-g(\theta)$ at the bottom left to $u_{\mathrm{in,2}} = 0 $ at the top left corner (and faster than linear) is a suitable choice for the analysis, which requires the velocity datum in $\bH^{1/2}(\Gi\cup\Gw \cup\Gm)$.
    
    As usual in membrane filtration processes, there are several orders of magnitude difference between the inlet and permeate flow velocities. More precisely, relating \eqref{model-eq4} and \eqref{model-eq9} we have the following inequality:
	\begin{equation}
		\label{eq:inequality-inlet-membrane}
		0\leq g(\theta) \ll \vert \buin\vert.
	\end{equation}
	Furthermore, and motivated by mass conservation properties of the flow, it is well-known that the inflow velocity \eqref{model-eq4}, the membrane filtration with assumption \eqref{eq:inequality-inlet-membrane}, the wall conditions \eqref{model-eq6*}, and the outflow boundary condition \eqref{model-eq6} are related by the mass flow rate through inlet and outlets 
 \begin{equation}
     \label{eq:conservation}
    - \int_{\Gi}\rho_0 \buin\cdot\bn =  \int_{\Gm\cup\Go} \rho_0\bu\cdot\bn \qquad \text{and}\qquad 0\leq  \int_{\Gm}  g(\theta) \leq \int_{\Go} \bu\cdot\bn,  
 \end{equation}
	where we are also assuming that the fluid density is a positive constant. 
	\subsection{Weak formulation and preliminary properties} \label{subsec:weak_formulation}
	For  the forthcoming analysis, instead of \eqref{eq:inequality-inlet-membrane} we can simply assume that there exist positive constants $0<g_1\leq g_2$ such that 
	\begin{equation}\label{eq:assumption-g}
		g_1 \leq g(s) \leq g_2 \qquad \forall s \in \mathbb{R}^+,\end{equation}
	and note that, in practical applications, $g$ is a linear function of concentration.

Note that the Cauchy pseudostress associated with the fluid is defined as 
\[ \bsig := \mu_0\nabla\bu - p\bI,\]
where $\bI$ is the identity tensor in $\mathbb{R}^{2\times 2}$. Note also that the traction vector along the boundary, $\bsig\bn$ can be decomposed into its normal and tangential parts as follows 
\[ \bsig\bn = (\bsig\bn \cdot \bn)\bn + %\sum_{j=1}^{n-1} 
(\bsig\bn \cdot \boldsymbol{t})\boldsymbol{t}. \]

On the permeable sub-boundary $\Gm$ we define the following quantity 
\begin{equation}
	\label{eq:lagmult_lambda}
	\lambda:=-(\bsig\bn) \cdot \bn, 
\end{equation}
which is treated as a Lagrange multiplier. We then proceed to define the following functional spaces for fluid velocity, pressure, the Lagrange multiplier, and concentration,  respectively 
\begin{linenomath}
	\begin{gather*}
		%	\bH&:=\{\bv\in \bH^1(\Omega)\;:\; \bv\cdot\bn=u_{in},\,\bv\times\bn=\boldsymbol{0} \text{ on }\Gamma_{\text{in}}\cup\Gm,\,\bv=\boldsymbol{0} \text{ on }\Gamma_D\},\\
		\bH:=\{\bv\in \bH^1(\Omega)\;:\;\bv\cdot\bt = 0 \text{ on }\Gm,\,\bv=\buin\text{ on }\Gi,\,\bv=  \boldsymbol{0} \text{ on }\Gw\},\\  
		\bH_0:=\{\bv\in \bH^1(\Omega)\;:\;\bv\cdot\bt = 0 \text{ on }\Gm,\,\bv=\boldsymbol{0} \text{ on }\Gi\cup\Gw\},\qquad   Q:=\L^2(\Omega),\qquad 
		W:=\H^{-1/2}(\Gm), \\ 
		%&Y:=\{\bv\in H^1(\Omega)^2\;:\; \bv=\boldsymbol{0} \text{ on }\Gamma_{\text{in}}\cup\Gamma_s,\; \bv\cdot\boldsymbol{t}=0\text{ on }\Gm\}, 
		%	Z&:=\{\tau\in \H^1(\Omega)\;:\;  \tau=\theta_{in} \text{ on }\Gamma_{\text{in}} \},\\
		Z:=\{\tau\in \H^1(\Omega)\;:\;  \tau=\theta_{\mathrm{in}} \text{ on }\Gi \}, \qquad 
		Z_0:=\{\tau\in \H^1(\Omega)\;:\;  \tau=0 \text{ on }\Gi \},
	\end{gather*}
\end{linenomath}
where the boundary specifications in the spaces $\bH,\,\bH_0,\,Z,\,Z_0$ are understood in the sense of traces.

By testing the first equation of \eqref{eq:coupled-model} against $\bv\in \bH_0$, integrating by parts, using  boundary conditions \eqref{model-eq6}--\eqref{model-eq6*} and \eqref{eq:lagmult_lambda} we obtain
\begin{equation}\label{pairing-1}
	\int_\Omega\mu_0\nabla\bu:\nabla \bv-\int_\Omega p  \nabla\cdot\bv+\int_\Omega \rho_0\left(\bu\cdot\nabla \bu\right)\cdot\bv + \langle\lambda,\bv\cdot\bn\rangle_{\Gm}=0.
\end{equation}
Here $\langle\cdot,\cdot\rangle_{\Gm}$ denotes the duality pairing between $\H^{-1/2}(\Gm)$ and its dual $\H^{1/2}(\Gm)$, with respect to the $L^2(\partial\Omega)$-norm. 
As usual, the incompressibility constraint is written weakly as 
\begin{equation*}
%	\label{eq:incompress-weak}
	\int_\Omega q \nabla \cdot \bu = 0 \quad \forall q\in L^2(\Omega).
\end{equation*}
On the other hand, using the incompressibility condition we can rewrite  \eqref{eq:coupled-model} as
\begin{equation}
	\label{eq:rewrite-convection-diffusion}
	\nabla\cdot\left(\theta\bu - D_0\nabla\theta\right)=0\quad \text{in}\;\O.
\end{equation}
Then, testing \eqref{eq:rewrite-convection-diffusion} against $\tau\in Z_0$, integrating by parts and using  \eqref{model-eq5} and \eqref{model-eq7} we obtain
\[
D_0\int_\O\nabla\theta\cdot\nabla\tau - \int_\O \theta (\bu\cdot\nabla\tau) + \langle \theta(\bu\cdot\bn),\tau\rangle_{\Go} =0 \quad \forall \tau \in Z_0,
\]
where the well-definedness of the last term on the left-hand side is addressed later in \eqref{eq:continuity-widetilde_c}.  

We remark that the permeability condition \eqref{model-eq9} is imposed weakly as follows
\begin{equation}\label{pairing-2}
	\langle\xi,\bu\cdot\bn\rangle_{\Gm} = \langle  \xi, g(\theta)\rangle_{\Gm}\quad\forall \xi\in W,
\end{equation}
whereas the zero tangential velocity condition \eqref{model-eq8} is imposed strongly on the velocity space. Furthermore, we introduce the following bilinear and trilinear forms
\begin{linenomath}
	\begin{gather*}
%		\label{eq:formas_bilineales_continuas}
		a(\bu,\bv):=\mu_0\int_\Omega \nabla\bu:\nabla\bv,\qquad 
		\widetilde{a}(\bw;\bu,\bv):=\rho_0\int_\Omega (\bw\cdot\nabla\bu)\cdot\bv,\qquad c(\theta,\tau):=D_0\int_\Omega \nabla \theta\cdot\nabla\tau, \\
		b(\bv,(q,\xi)):=-\int_\Omega q \nabla\cdot\bv+\langle\xi,\bv\cdot\bn\rangle_{\Gm}, 
		\qquad 
		\widetilde{c}(\bw;\theta,\tau):=-\int_\Omega \theta(\bw\cdot\nabla \tau)+\langle\theta(\bw\cdot \bn),\tau\rangle_{\Go}.\nonumber
	\end{gather*}
\end{linenomath}

Thanks to the assumed regularity of the inflow velocity and concentration, it is possible to prove that the   velocity and concentration extension functions (liftings) $\bU \in \bH^1(\Omega)$ and $\Theta \in \H^1(\Omega)$, respectively,  exist and are well defined (see, e.g. \cite{arndt2020existence,lorca1999initial}), where  
they satisfy (in the sense of traces) 
\begin{equation}
    \label{eq:102}
\bU = \buin \;\text{on}\; \Gi, \quad 
\bU = \boldsymbol{0} \;\text{on}\;\Gw, \quad \bU \cdot \bt = 0 \;\text{on}\; \Gm, \quad \div\,\bU = 0 \;\text{in}\;\Omega, \quad 
\Theta = \theta_{\mathrm{in}} \;\text{on}\;\Gi.\end{equation}
With them, we have that a weak solution for the coupled model is defined as $(\bu,(p,\lambda),\theta)\in \bH\times (Q\times W) \times Z$ with 
\[\bu = \bU + \bu_0,\qquad \theta = \Theta + \theta_0,\] 
and where $(\bu_0,(p,\lambda),\theta_0)\in \bH_0\times (Q\times W) \times Z_0$
solves
\begin{equation}
	\label{eq:cross-flow-fv1}
	\begin{aligned}
		a( \bU + \bu_0,\bv)+ \widetilde{a}( \bU + \bu_0; \bU + \bu_0,\bv) + b(\bv,(p,\lambda)) &=0 &\forall \bv\in \bH_0,\\
		b( \bU + \bu_0,(q,\xi))&= \langle\xi, g(\Theta + \theta_0)\rangle_{\Gm} &\forall (q,\xi)\in Q\times W,\\
		c(\Theta + \theta_0,\tau) + \widetilde{c}(\bU + \bu_0;\Theta + \theta_0,\tau)&=0&\forall \tau\in Z_0.
	\end{aligned}
\end{equation}
Note that if prescribed boundary conditions for both $\bu$ and $\theta$ are considered in $\Gm$ (that is, without a coupling effect), then we fall into a typical formulation of Boussinesq equations. On the other hand, note that the nonlinearity inherited by the boundary condition \eqref{model-eq7} is present in the trilinear form $\widetilde{c}$ trough a contribution in $\Go$. 

Next we stress that the bilinear and trilinear forms considered in the above weak formulation are uniformly bounded. It suffices to apply  Hölder's inequality, Sobolev embeddings and trace inequalities (see, for instance, \cite[Section 1.1]{temam2001navier}):
%or \cite[Section 9.2]{brezis2011functional}):
\begin{linenomath}
	\begin{subequations}
		\begin{align}
			\label{eq:cont-a}   \vert a(\bu,\bv)\vert&\lesssim \Vert \bu\Vert_{1,\O}\Vert \bv\Vert_{1,\O}&\bu,\bv\in \bH^1(\O),\\
			\vert b(\bv,(q,\xi))\vert &\lesssim \Vert\bv\Vert_{1,\O}(\Vert q\Vert_{0,\O}+\Vert \xi\Vert_{-1/2,\Gm})&\bv\in\bH^1(\O),q\in L^2(\O),\xi\in \H^{-1/2}(\Gm),\label{eq:cont-b}\\
			\vert c(\psi,\tau)\vert&\lesssim \Vert \psi\Vert_{1,\O}\Vert \tau\Vert_{1,\O} &\psi,\tau\in \H^1(\O), \label{eq:cont-c}
			\\
			|\widetilde{a}(\bw;\bu,\bv)|&\lesssim \Vert\bw\Vert_{1,\Omega}\,\Vert\bu\Vert_{1,\Omega}\,\Vert\bv\Vert_{1,\Omega},&\bw,\bu,\bv\in \bH^1(\O)\label{eq:continuity-widetilde_a}\\
			|\widetilde{c}(\bw;\theta,\tau)|&\lesssim \Vert\bw\Vert_{1,\Omega}\,\Vert\tau\Vert_{1,\Omega}\, \,\Vert\theta\Vert_{1,\Omega},&\bw\in\bH^1(\O),\theta,\tau\in \H^1(\O)\label{eq:continuity-widetilde_c}.
\end{align}\end{subequations}\end{linenomath}
%for all $\bu,\bw\in \H$, $\bv\in \H_0$, and for all $\tau\in Z_0$ and $\theta \in Z$.
In particular, for \eqref{eq:continuity-widetilde_c} we have used that
%\begin{subequations}
$$
		\left|\int_{\Go}\theta(\bw\cdot\bn)\tau\right|\leq \Vert\tau\Vert_{L^4(\Go)}\, \Vert\bw\Vert_{0,\Go}\,\Vert\theta\Vert_{L^4(\Go)}\leq \Vert\tau\Vert_{L^4(\partial\Omega)}\, \Vert\bw\Vert_{0,\partial\Omega}\,\Vert\theta\Vert_{L^4(\partial\Omega)}.
		%\label{eq:coer-c}
$$
We also note that thanks to the vector and scalar forms of Poincar\'e inequality, we have, respectively, the ellipticity for the bilinear forms $a(\cdot,\cdot)$ and $c(\cdot,\cdot)$ as
\begin{linenomath}
	\begin{subequations}
		\begin{align}
			\label{eq:coer-a}    |a(\bv,\bv)|&\gtrsim \Vert \bv\Vert_{1,\O}^2&\forall\bv\in\bH^1(\O),\\
			\label{eq:coer-c}   |c(\tau,\tau)|&\gtrsim \Vert \tau\Vert_{1,\O}^2  &\forall\tau\in\H^1(\O).    
\end{align}\end{subequations}\end{linenomath}

Consider a fixed $\zeta\in Z$ and denote  by $\bX^g$ the following subspace of $\bH$ associated with the bilinear form  $b(\cdot,(\cdot,\cdot))$ 
\begin{align}
	\label{eq:X}
	\bX^g&:=\left\{\bv\in\bH: b(\bv,(q,\xi))=\langle \xi, g(\zeta)\rangle_{\Gm} \ \forall (q,\xi)\in Q\times W \right\}\nonumber\\
	& =\left\{\bv\in\bH: \nabla\cdot\bv=0 \text{ in }\O,  \quad \bv\cdot\bn =g(\zeta) \text{ on }\Gm\right\},\\
	& =\left\{\bv\in\bH^1(\Omega): \nabla\cdot\bv=0 \text{ in }\O, \quad \bv =\buin \text{ on }\Gi, \quad \bv\cdot\bn= g(\zeta) \text{ on }\Gm, \quad \bv = \bzero  \text{ on }\Gw\right\}.	\nonumber
\end{align}
For the advection term we use the well-known identity 
\[\int_\Omega (\bw \cdot \nabla \theta) \tau + \int_\Omega (\bw \cdot \nabla \tau) \theta  = -\int_\Omega (\nabla\cdot \bw) \theta \tau + \int_{\partial\Omega}(\bw \cdot \bn) \theta \tau, \]
to readily obtain 
\begin{equation}\label{tc-form}
	\widetilde{c}(\bw;\tau,\tau)=\frac{1}{2}[(\bw\cdot\bn,\tau^2)_{\Go} -(\bw\cdot\bn,\tau^2)_{\Gm} ]  \qquad \forall \bw\in\bX^g,\; \forall \tau\in Z_0,
\end{equation}
and thanks to the inlet boundary condition, the property \eqref{eq:inequality-inlet-membrane}, and a simple conservation argument following \eqref{eq:conservation}, it follows that $\int_{\Go}  \bw\cdot\bn \geq \int_{\Gm} \bw\cdot\bn $. Hence,
\begin{equation}\label{eq:ct-pos}
	\widetilde{c}(\bw;\tau,\tau)\geq 0, \qquad \bw\in\bX^g, \tau\in Z_0.
\end{equation}

%%%%%%%%%%%%%%%%%%%%%%%%%%%%%%%%%%%%%%%%%%%%%%%%%%%%
\section{Well-posedness of the continuous problem}\label{sec:wellp}
For the analysis of existence and uniqueness of solution we use a fixed-point argument separating the solution between two uncoupled problems. First consider the Navier--Stokes equations, which consist in finding, for a given $\zeta = \zeta_0 + \Theta$ (with $\zeta_0\in Z_0$), the tuple  $(\bu_0+\bU,p,\lambda) \in \bH\times Q\times W$ such that   
\begin{linenomath}
	\begin{subequations}
		\begin{align}
			\label{eq:ns1}   a(\bu,\bv)  + \widetilde{a}(\bu;\bu,\bv) 
   + b(\bv,(p,\lambda)) & =0
   \quad \forall \bv \in \bH_0,\\
			\label{eq:ns2}    b(\bu,(q,\xi)) & = \langle \xi, g(\zeta) \rangle_{\Gm} \quad \forall (q,\xi)\in Q\times W.
		\end{align}
\end{subequations}\end{linenomath}
Note also that if $(\bw,p,\lambda)\in \bH\times Q\times W$ is a solution to \eqref{eq:ns1}-\eqref{eq:ns2} then $\bw$ is in the space $\bX^g$. 

Secondly, consider the uncoupled advection--diffusion equation in weak form: For a given advecting velocity $\bw =\bw_0 + \bU \in \bX^g$ (with $\bw_0 \in \bH_0$), find $\theta_0\in Z_0$ such that  
\begin{equation}
	\label{eq:diff}    c(\theta_0,\tau) + \widetilde{c}(\bw;\theta_0,\tau) = -c(\Theta,\tau) - \widetilde{c}(\bw; \Theta,\tau) \quad \forall \tau \in Z_0.
\end{equation}

%%%%%%%%%%%%%%%%%%%%%%%
\subsection{Well-posedness of the Navier--Stokes equations with membrane boundary conditions}\label{sec:wellp-ns}
In order to address the unique solvability of \eqref{eq:ns1}-\eqref{eq:ns2}, we use a linear Stokes problem with membrane boundary conditions, the Banach fixed-point theorem, and the  Babu\v{s}ka--Brezzi theory for saddle-point problems. %\cite{gatica2021further}. 
For this we follow the analysis from \cite{cocquet2021error}. 
Let us consider the problem of finding, for a given $\zeta\in Z$ and a given $\bw\in\bH_0$, the tuple  $(\bu_0,p,\lambda) \in \bH_0\times Q\times W$ such that 
\begin{linenomath}
	\begin{subequations}
		\begin{align}
			\label{eq:stokes1}   a(\bu_0,\bv)  + b(\bv,(p,\lambda)) & = - a(\bU,\bv) - \widetilde{a}(\bw+\bU;\bw+\bU,\bv) \quad \forall \bv \in \bH_0,\\
			\label{eq:stokes2}    b(\bu_0,(q,\xi)) & = \langle \xi, g(\zeta) \rangle_{\Gm} \quad \forall (q,\xi)\in Q\times W.
		\end{align}
\end{subequations}\end{linenomath}
In order to show that this linear Stokes system is well-posed we follow  arguments similar to \cite[Lemma 3.1]{verfurth86}, \cite{ern2004applied,layton1999weak}.  We start with  the following result (the proof is carried out in a standard way, but we present it for sake of completeness). 
\begin{lemma}\label{lem:B-inf}
	The following inf-sup condition holds 
	\[
	\sup_{\bv\in \bH_0,\bv\neq\boldsymbol{0}}\frac{
		 b(\bv,(q,\xi)) 
		%b(\bv,(q,\xi))
	}{\Vert\bv\Vert_{1,\Omega}}\gtrsim \Vert q\Vert_{0,\Omega}+\Vert \xi\Vert_{-1/2,\Gm} \quad \forall (q,\xi)\in Q\times W. %\label{eq:inf-sup-cont}
	\]
\end{lemma}
\begin{proof}
Thanks to the Riesz representation theorem, for a given $\xi \in W$ there exists $\tilde{\xi} \in \H^{1/2}(\Gm)$ such that $\|\tilde{\xi}\|_{1/2,\Gm}=\|\xi\|_{-1/2,\Gm}$. 
	For a given pair $(q,\xi)\in Q\times W$, let us consider the following auxiliary  Stokes problem with mixed boundary conditions 
	\begin{linenomath}\begin{align}
			-\boldsymbol{\Delta} \hat{\bv} + \nabla\zeta  = \boldsymbol{0} \quad \text{and}\quad \nabla\cdot \hat{\bv} &= q \quad  \text{in } \Omega,\nonumber\\
			(\boldsymbol{\nabla}\hat{\bv} - \zeta\bI)\bn & = \boldsymbol{0} \quad \text{on } \Go,\label{aux1}\\
   \hat{\bv} & = \bzero  \quad \text{on } \Gi \cup \Gw,\nonumber \\
			\hat{\bv}& = \tilde{\xi}\bn \quad \text{on } \Gm.\nonumber
		\end{align}    
	\end{linenomath}
For the compatibility of the Dirichlet data in \eqref{aux1}, we can proceed similarly as in \cite{lions1969quelques,grisvard2011elliptic} (see also \cite{le1993steady}) and use a cut-off and mollification argument. 
%We present it here for completeness. 
Let us first define an extension of $\tilde{\xi}$ on $\Gamma_* := \Gm \cup \Gi \cup \Gw$ as follows
\[ \tilde{\xi}_0(\bx) := \begin{cases}
    \tilde{\xi}(\bx) & \bx \in \Gm,\\
    0 & \bx \in \Gamma_*\setminus \Gm,
\end{cases}\]
and note that the extension is linear in $\tilde{\xi}$. Also, $\tilde{\xi}_0 \in \mathrm{L}^2(\Gamma_*)$, but it is not necessarily in $\mathrm{H}^{1/2}(\Gamma_*)$ because of the jump across the left-bottom corner belonging to $\partial\Gm$. Next we consider a smooth cut-off function $\eta_{\epsilon} \in C^\infty(\Gamma_*)$ satisfying $0 \leq \eta_{\epsilon} \leq 1$, $\eta_\epsilon \equiv 1$ on the interior of $\Gm$, and $\text{supp}(\eta_{\epsilon})\subset \Gamma_*$. Its transition from 1 to 0 happens in a small layer of width $\epsilon$ near the aforementioned corner in $\partial\Gm \subset \Gamma_*$. In addition, we denote by $\rho_{\epsilon}$ a standard smooth mollifier
 defined on local charts of $\Gamma_*$, and define the function 
 \[ \xi_{\epsilon}^*:= \eta_{\epsilon}(\tilde{\xi}_0 * \rho_{\epsilon}),\]
 where the convolution smooths out the jump across the left-bottom corner in $\partial\Gm$. Then we have $\xi_{\epsilon}^* \in \mathrm{H}^{1/2}(\Gamma_*)$, $\text{supp}(\xi_{\epsilon}^*)\subset \Gamma_*$, $\xi_{\epsilon}^* \equiv \tilde{\xi}$ on the interior of $\Gm$ (away from the $\epsilon$-transition layer), and $\xi_{\epsilon}^*$ vanishes near $\partial\Gamma_*$. Therefore $\xi_{\epsilon}^* \in \mathrm{H}^{1/2}_{00}(\Gamma_*)$.
 
 We further notice that the map $\tilde{\xi}\mapsto \xi_{\epsilon}^*$ is linear and bounded
 \begin{equation}\label{aux3} \| \xi_{\epsilon}^*\|_{1/2,00;\Gamma_*} \leq C_{\epsilon} \| \tilde{\xi}\|_{1/2,\Gm} =  C_{\epsilon} \|\xi\|_{-1/2,\Gm}, \end{equation}
 with $C_{\epsilon}>0$ a constant depending only on $\epsilon$ and on the geometry of $\Gm$ and $\Gamma_*$.
 
 Then, the  boundary conditions in \eqref{aux1} are rewritten in a well-defined manner as 
 \begin{equation}\label{aux2} 
 (\boldsymbol{\nabla}\hat{\bv} - \zeta\bI)\bn  = \boldsymbol{0} \quad \text{on } \Go, \qquad 
  \hat{\bv} = \xi_{\epsilon}^*\bn \quad \text{on } \Gamma_* = \partial\Omega \setminus \Go,\end{equation}
with $\xi_{\epsilon}^* \bn \in \bH^{1/2}_{00}(\Gamma_*)$, and thanks    
	%By elliptic regularity 
to \cite{girault1979finite}, we can assert that there exists a unique velocity solution to \eqref{aux1} (with mixed boundary conditions   as in \eqref{aux2}), for which it is straightforward to verify that there holds  
	\begin{subequations}
		\begin{gather}\label{v1a}
			\nabla\cdot\hat{\bv} =  q, \quad (\hat{\bv}\cdot\bn)|_{\Gm} = \tilde\xi, \quad (\hat{\bv}\cdot\bt)|_{\Gm} = 0,\\ 
			\|\hat{\bv}\|_{1,\Omega}\lesssim \Vert q\Vert_{0,\Omega}+\Vert \xi\Vert_{-1/2,\Gm}, \label{v1b}
		\end{gather}
	\end{subequations}
(where we have also used \eqref{aux3}). In this way we have that  $\hat{\bv} \in  \bH_0\setminus\{\boldsymbol{0}\}$, and thus we  can write 
	$$
		\sup_{\bv\in \bH_0,\bv\neq\boldsymbol{0}}\frac{b(\bv,(q,\xi))}{\Vert\bv\Vert_{1,\Omega}}\geq 
		\frac{b(\hat{\bv},(q,\xi))}{\Vert\hat{\bv}\Vert_{1,\Omega}}
		=  \frac{\Vert q\Vert^2_{0,\Omega}+\Vert \xi\Vert^2_{-1/2,\Gm}}{\Vert\hat{\bv}\Vert_{1,\Omega}}
		\gtrsim \Vert q\Vert_{0,\Omega}+\Vert \xi\Vert_{-1/2,\Gm},
	$$
	where we have employed \eqref{v1a} and \eqref{v1b}. 
\end{proof}

In the context of the fixed-point analysis of \eqref{eq:ns1}-\eqref{eq:ns2}, for a given $\zeta = \zeta_0 + \Theta \in Z$ we define the linear functional $G_\zeta\in (Q\times W)'$ as follows
\[\langle G_\zeta, (q,\xi)\rangle := \langle \xi, g(\zeta)\rangle_{\Gm} \quad \forall (q,\xi)\in Q\times W.\]
%(which, thanks to \eqref{eq:assumption-g} satisfies $\|G_\zeta\|_{(Q\times W)'}= g_2$). 
Similarly, for a given $\bw_0\in \bH_0$ we define the linear functional $F_{\bw_0,\bU}\in \bH_0'$:
$$
\bv \mapsto \langle F_{\bw_0,\bU}, \bv\rangle := - \tilde{a}(\bw_0 + \bU, \bw_0 + \bU, \bv) - a(\bU,\bv),
$$
where $\bU\in\bH^1(\Omega)$ is the divergence-free lifting defined in \eqref{eq:102}. 
Then, there holds (see \cite[Lemma 16]{cocquet2021error}) 
\begin{equation}
	\label{eq:103}
	\Vert \bU\Vert_{1,\O}\lesssim \Vert \buin\Vert_{1/2,\Gi}.
\end{equation}

\begin{lemma}\label{lem:stokes}
    For known liftings $\bU, \Theta$ and given $\bw_0\in\bH_0$ and $\zeta_0\in Z_0$, there exists a unique $(\bu_0,p,\lambda)\in\bH_0\times Q\times W$ such that
 \begin{linenomath}
 	\begin{subequations}
 		\begin{align}
 			\label{eq:104a}   a(\bu_0,\bv)  + b(\bv,(p,\lambda)) & = F_{\bw_0,\bU}(\bv) \quad \forall \bv \in \bH_0,\\
 			\label{eq:104b}    b(\bu_0,(q,\xi)) & = G_\zeta(q,\xi) \quad \forall (q,\xi)\in Q\times W.
 		\end{align}
 \end{subequations}\end{linenomath}
Moreover, the following estimates hold  
\begin{subequations}\begin{align}
	\label{eq:107}
	\Vert\bu_0\Vert_{1,\O}&\leq \frac{1}{\underline{\alpha}}\left[\Vert F_{\bw_0,\bU}\Vert_{\bH_0'} +\frac{\underline{\alpha}+\|a\|}{\beta}g_2\right],\\
	\label{eq:108}
	\Vert(p,\lambda)\Vert_{Q\times W}&\leq\frac{1}{\beta}\left[\left(1+\frac{\|a\|}{\underline{\alpha}}\right)
 \Vert F_{\bw_0,\bU}\Vert_{\bH_0'} 
 +\frac{\|a\|(\underline{\alpha}+\|a\|)}{\underline{\alpha}\beta}g_2\right].
\end{align}\end{subequations}
\end{lemma}
\begin{proof}
Note that the linear functionals are bounded. Indeed, we have 
$$
	|G_{\zeta}(\xi)|\leq \Vert\xi\Vert_{-1/2,\Gm}\Vert g(\zeta_0+\Theta)\Vert_{1/2,\Gm}\leq g_2\Vert \xi\Vert_{-1/2,\Gm} \qquad \forall (q,\xi)\in Q\times W,
$$
which, owing to \eqref{eq:assumption-g},  implies that 
\begin{equation}
	\label{eq:106}
	\Vert G_{\zeta}\Vert_{(Q\times W)'} = g_2.  
\end{equation}
Also, after using triangle inequality together with the boundedness from \eqref{eq:cont-a} and  \eqref{eq:continuity-widetilde_a}, we obtain 
$$
\begin{aligned}
	|F_{\bw_0,\bU}(\bv)|&= |-\tilde{a}(\bw_0 + \bU, \bw_0 + \bU, \bv) - a(\bU,\bv)|\\
	&\leq C_i\rho_0\Vert \bw_0 +\bU\Vert_{1,\Omega}^ 2\Vert \bv\Vert_1 + \mu_0\Vert \bU\Vert_{1,\O}\Vert \bv\Vert_{1,\O}\\
	&\leq \left[C_i\rho_0\left(\Vert\bw_0\Vert_{1,\O}^ 2 +\Vert \bU\Vert_{1,\O}^2\right) + \mu_0\Vert\bU\Vert_{1,\O} \right]\Vert\bv\Vert_{1,\O},
\end{aligned}
$$
and due to \eqref{eq:103}, there holds 
$$
\Vert F_{\bw_0,\bU}\Vert_{\bH_0'}=C_i\rho_0\left(\Vert\bw_0\Vert_{1,\O}^ 2 +\Vert \bu_{\mathrm{in}}\Vert_{1/2,\Gi}^2\right) + \mu_0\Vert \bu_{\mathrm{in}}\Vert_{1/2,\Gi},
$$
where $C_i>0$ is the continuity constant of the Sobolev embedding used in \eqref{eq:continuity-widetilde_a}. 

Using Lemma \ref{lem:B-inf} with  inf-sup  constant $\beta$ depending on the trace inequality constant, embedding theorems and elliptic regularity, the coercivity of the bilinear form $a(\cdot,\cdot)$ \eqref{eq:coer-a} with constant $\underline{\alpha} = \frac{\mu_0}{2}\min\{\frac{1}{C_P^2},1\}$  (where $C_P$ denotes the Poincar\'e constant), the continuity of $a(\cdot,\cdot)$ \eqref{eq:cont-a} with constant $\|a \| \leq \max\{1,\mu_0\}$, and the continuity of the bilinear form $b(\cdot,(\cdot,\cdot))$ with constant $\|b\| \leq 1$ in \eqref{eq:cont-b}, the Babu\v{s}ka--Brezzi theory (see, e.g., \cite[Theorem 2.34]{ern2004applied}) guarantees that there exists a unique tuple $(\bu_0,p,\lambda)$ solution of \eqref{eq:104a}--\eqref{eq:104b}, which also satisfies the  continuous dependence on data. 
\end{proof}

Let us remark that the unique Stokes velocity from Lemma~\ref{lem:stokes} is in the space $\bX^g$.  We now derive a fixed-point strategy for \eqref{eq:ns1}--\eqref{eq:ns2}. Let us define the map 
$$
\mathcal{J}:\bH\to \bH\times Q\times W, \quad \bw \mapsto 
	\mathcal{J}(\bw)=(\mathcal{J}_1(\bw),\mathcal{J}_2(\bw),\mathcal{J}_3(\bw))=:(\bu_0+\bU,p,\lambda),
$$
where $(\bu_0,p,\lambda)$ is the unique solution to \eqref{eq:104a}--\eqref{eq:104b} and $(\bu_0+\bU,p,\lambda)$ includes the non-homogeneous boundary condition. We focus our attention on the first component of the mapping, i.e.  $\mathcal{J}_1:\bH\rightarrow\bH$. The following results collect the required properties for the application of the Banach fixed-point theorem on $\mathcal{J}_1$, and hence the existence and uniqueness of a solution to \eqref{eq:ns1}--\eqref{eq:ns2}. 
\begin{lemma}\label{lem:ns-map}
 Consider the following closed ball of $\bH$
 \[\bM_{R_0} = \{ \bv \in \bX^g: \|\bv\|_{1,\Omega}\leq R_0\}. \]
 Assume that the data (in particular, $\|\bu_{\mathrm{in}}\|_{1/2,\Gi}$) are sufficiently small so that 
 \begin{equation}\label{eq:assumption-uin}  0<4R_0<1-\sqrt{1-4(\Vert \buin\Vert_{1/2,\Gi}^2 + g_2 + 2\Vert \buin\Vert_{1/2,\Gi})}. 
 %\cred{[\text{Check $R_0$ choice}]}
 \end{equation}
 Then $\mathcal{J}_1(\bM_{R_0}) \subseteq \bM_{R_0}$.
\end{lemma}
\begin{proof}
    Let us fix $R_0> 0$ and consider $\bw\in \bM_{R_0}$. We have, thanks to the definition of $\mathcal{J}_1$, triangle inequality, and \eqref{eq:107}, that 
    $$
    \begin{aligned}        \Vert\mathcal{J}_1(\bw)\|_{1,\Omega}&=\Vert\bu_0 + \bU \Vert_{1,\O}\\
    &\leq \frac{1}{\underline{\alpha}}\left[C_i\rho_0\left(\Vert\bw_0\Vert_{1,\O}^ 2 +\Vert \bu_{\mathrm{in}}\Vert_{1/2,\Gi}^2\right) + \mu_0\Vert \bu_{\mathrm{in}}\Vert_{1/2,\Gi}\right] %\\
    %&\hspace{2cm}
    + \frac{\|a\|(\underline{\alpha}+\|a\|)}{\underline{\alpha}\beta}g_2+ \Vert \buin\Vert_{1/2,\Gi}\\
    &\lesssim \Vert\bw\Vert_{1,\O}^2 + \Vert \buin\Vert_{1/2,\Gi}^2 + g_2+ 2\Vert \buin\Vert_{1/2,\Gi}\\
    &\lesssim R_0^ 2+ \Vert \buin\Vert_{1/2,\Gi}^2 + g_2 + 2\Vert \buin\Vert_{1/2,\Gi},
    \end{aligned}
    $$
    where the hidden constant depends on $\underline{\alpha}, \|a\|, \beta, C_i$, and $\rho_0$. 
    After elementary algebraic manipulations we can assert that the right-hand side above is smaller or equal than $\frac{R_0}{2}$ if \eqref{eq:assumption-uin} holds. 
\end{proof}

\begin{lemma}\label{lem:ns-lip}
 There exists a positive constant $L_{\mathcal{J}_1}$, depending only on data (in particular, on the inlet velocity $\|\bu_{\mathrm{in}}\|_{1/2,\Gi}$), such that 
 \[\|\mathcal{J}_1(\bw_1)- \mathcal{J}_1(\bw_2)\|_{1,\Omega} \leq L_{\mathcal{J}_1} \|\bw_1 - \bw_2\|_{1,\Omega} \qquad \forall \bw_1,\bw_2 \in \bM_{R_0}.\]
\end{lemma}
\begin{proof}
Given $\bw_1,\bw_2 \in \bM_{R_0}$, let us consider the two well-posed Stokes problems \eqref{eq:stokes1}-\eqref{eq:stokes2} for each fixed velocity and giving the unique solutions 
\[(\bu_{01}+\bU,p_1,\lambda_1) = \mathcal{J}(\bw_1), \quad (\bu_{02}+\bU,p_2,\lambda_2) = \mathcal{J}(\bw_2) ,\]
respectively. %Note that $\bu_{01}+\bU,\bu_{02}+\bU \in \bX^g$ and therefore subtracting 
Subtracting the corresponding first and second equations in these problems and noticing that $G$ does not depend on $\bw_1,\bw_2$, we arrive at 
\begin{subequations}\label{diff-}
  \begin{align}
 \label{diff-a}	   a(\bu_{01}-\bu_{02},\bv)  + b(\bv,(p_1-p_2,\lambda_1-\lambda_2)) & = F_{\bw_1,\bU}(\bv) -F_{\bw_2,\bU}(\bv) \quad \forall \bv \in \bH_0,\\
 			   b(\bu_{01}-\bu_{02},(q,\xi)) & = 0 \quad \forall (q,\xi)\in Q\times W.
 		\end{align}\end{subequations}
Regarding  the right-hand side of \eqref{diff-a}, note now that 
\begin{align*}
|F_{\bw_1,\bU}(\bv) &-F_{\bw_2,\bU}(\bv) |  = \biggl|\int_\Omega [(\bw_2+\bU)\cdot \nabla (\bw_2+\bU)-(\bw_1+\bU)\cdot \nabla (\bw_1+\bU)]\cdot \bv\biggr| \\
& \leq \int_\Omega |(\bw_2 - \bw_1)\cdot \nabla (\bw_2+\bU) \cdot \bv | + \int_\Omega |(\bw_1+\bU) \cdot \nabla (\bw_2-\bw_1) \cdot \bv|,\\
& \lesssim (\|\bw_1\|_{1,\Omega} + \|\bw_2\|_{1,\Omega} + \|\buin\|_{1/2,\Gi})\|\bw_1-\bw_2\|_{1,\Omega}\|\bv\|_{1,\Omega}.
\end{align*}

On the other hand, in \eqref{diff-} we can use as test functions $\bv =\bu_{01}-\bu_{02} $, $q = p_1-p_2$, $\xi = \lambda_1-\lambda_2$ and subtract the two equations  to obtain 
\[
    a(\bu_{01}-\bu_{02},\bu_{01}-\bu_{02}) = F_{\bw_1,\bU}(\bu_{01}-\bu_{02}) -F_{\bw_2,\bU}(\bu_{01}-\bu_{02}). 
\]
Finally, we use the definition of $\mathcal{J}$, the two previous results, and the coercivity of  $a(\cdot,\cdot)$ to get 
\begin{align*}
    \|\mathcal{J}_1(\bw_1)- &\mathcal{J}_1(\bw_2)\|^2_{1,\Omega}  = \|\bu_{01}+\bU-\bu_{02}-\bU\|^2_{1,\Omega} \\
    & \leq \frac{1}{\underline{\alpha}^2}  a(\bu_{01}-\bu_{02},\bu_{01}-\bu_{02}) \\
    & \leq  \frac{1}{\underline{\alpha}^2} | F_{\bw_1,\bU}(\bu_{01}-\bu_{02}) -F_{\bw_2,\bU}(\bu_{01}-\bu_{02})| \\
    & \lesssim (\|\bw_1\|_{1,\Omega} + \|\bw_2\|_{1,\Omega} + \|\buin\|_{1/2,\Gi}) \|\bw_1-\bw_2\|_{1,\Omega}\|\bu_{01}-\bu_{02}\|_{1,\Omega} \\
    & = (\|\bw_1\|_{1,\Omega} + \|\bw_2\|_{1,\Omega} + \|\buin\|_{1/2,\Gi}) \|\bw_1-\bw_2\|_{1,\Omega}\|\mathcal{J}_1(\bw_1)- \mathcal{J}_1(\bw_2)\|_{1,\Omega}.
\end{align*}
Then the desired result follows by dividing through $\|\mathcal{J}_1(\bw_1)- \mathcal{J}_1(\bw_2)\|_{1,\Omega}$,   recalling that $\bw_1,\bw_2\in \bM_{R_0}$, and choosing 
\begin{equation}\label{eq:choosing-L} L_{\mathcal{J}_1} =  \frac{1}{\underline{\alpha}^2}(2R_0 + \|\buin\|_{1/2,\Gi}). \end{equation}
\end{proof}

\begin{theorem}\label{th:ns}
Given $\zeta_0\in Z_0$, assume that the data is sufficiently small as in \eqref{eq:assumption-uin}, and in light of \eqref{eq:choosing-L}, further assume that $R_0$ is taken such as 
 \begin{equation}\label{eq:assumption-L}
L_{\mathcal{J}_1}= 2R_0 + \|\buin\|_{1/2,\Gi}  < \underline{\alpha}^2. 
%\cred{[\text{Is this the same $L$ used in Section \ref{sec:discrete-wellp}?}]}
 \end{equation}
 Then there exists a unique solution $(\bu_0+\bU,p,\lambda) \in \bH \times Q \times W$ to \eqref{eq:ns1}--\eqref{eq:ns2}. 
\end{theorem}
\begin{proof}
The result is a direct consequence of the well-definedness of $\mathcal{J}$ together with Lemmas~\ref{lem:ns-map} and \ref{lem:ns-lip}, and the fact that \eqref{eq:assumption-L} gives that  $\mathcal{J}$ is a contraction mapping. 
\end{proof}
Note that the proof of Theorem~\ref{th:ns} is also valid if in Lemmas~\ref{lem:ns-map}-\ref{lem:ns-lip} we take any $\bw\in \bH$ with $\|\bw\|_{1,\Omega}\leq R_0$, that is, we have not used that $\bw\in \bX^g$. This additional condition is required in the  analysis of unique solvability of the decoupled advection--diffusion, as stated next.

	%%%%%%%%%%%%%%%%%%%%%%%
	\subsection{Well-posedness of the advection--diffusion equation}\label{sec:wellp-ad}
	The unique solvability of problem \eqref{eq:diff} follows after using \eqref{eq:cont-c}, \eqref{eq:coer-c}, \eqref{eq:ct-pos} together with the Lax--Milgram lemma, which also gives  
	\begin{align}
		\label{eq:bound-}
		\|\theta_0\|_{1,\Omega} & \lesssim \|\bw\|_{1,\Omega}(1 + \|\Theta\|_{1,\Omega}) +  \|\Theta\|_{1,\Omega} \nonumber \\
  & \leq \frac{1}{\underline{\alpha}}\left[\Vert F_{\bw_0,\bU}\Vert_{\bH_0'} +\frac{\underline{\alpha}+\|a\|}{\beta}g_2\right] (1 + \|\theta_{\mathrm{in}}\|_{1/2,\Gi}) + \|\theta_{\mathrm{in}}\|_{1/2,\Gi},
	\end{align} 
	where we have used trace inequality and continuous dependence on data of both uncoupled problems. 
	
	%%%%%%%%%%%%%%%%%%%%
	\subsection{Fixed-point analysis for the coupled flow--transport problem}\label{sec:fixed-point}
	With the development above, we are now able to properly define the following solution operators 
	\[\tilde{S}: \bH \to Z, \quad \bw \mapsto \tilde{S}(\bw) := \theta_0 + \Theta, \]
	where $\theta_0$ is the unique solution of \eqref{eq:diff} (confirmed in Section~\ref{sec:wellp-ad}), and 
	\[S: Z \to \bH, \quad \zeta \mapsto S(\zeta) = (S_I(\zeta),S_{I\!I}(\zeta),S_{I\!I\!I}(\zeta)):= (\bu_0+\bU,p,\lambda), \]
	where $(\bu_0+\bU,p,\lambda)$ is the unique solution of \eqref{eq:ns1}-\eqref{eq:ns2} (established in Section~\ref{sec:wellp-ns}). 
	The nonlinear problem in weak form \eqref{eq:cross-flow-fv1} is therefore equivalent to the following fixed-point equation
	\begin{equation}
		\text{find $\bu= \bu_0 + \bU \in \bH$ such that} \   \bu = T(\bu),    \label{eq:fixed-point}
	\end{equation}
	where $T:\bH\to \bH$ is defined as $\bu \mapsto T(\bu) = ({S}_I\circ \tilde{S})(\bu_0 + \bU)$. 
	
	We proceed then to define the closed ball in $\bH$ 
	$\bM_{R_1} = \{\bw \in \bX^g: \|\bw\|_{1,\Omega} \leq   R_1\}$,  
	and assume that $R_1<1$, which (owing to Lemma~\ref{lem:ns-map} and \eqref{eq:bound-}) amounts to consider the assumption on the model data 
	\begin{equation}
		\label{eq:assumption-data}
	\begin{aligned}
	   & \max\left\{  R_0^ 2+ \Vert \buin\Vert_{1/2,\Gi}^2 + g_2 + 2\Vert \buin\Vert_{1/2,\Gi},\right. \\
    & \qquad \frac{1}{\underline{\alpha}}\left[
    %\Vert F_{\bw_0,\bU}\Vert_{\bH_0'} 
    C_i\rho_0(R_0^ 2 +\Vert \bu_{\mathrm{in}}\Vert_{1/2,\Gi}^2)%\right. \\
    %&\hspace{2cm}\left.
    \left.
    + \mu_0\Vert \bu_{\mathrm{in}}\Vert_{1/2,\Gi}
    +\frac{\underline{\alpha}+\|a\|}{\beta}g_2\right] (1 + \|\theta_{\mathrm{in}}\|_{1/2,\Gi}) + \|\theta_{\mathrm{in}}\|_{1/2,\Gi} \right\} <1.\end{aligned}
	\end{equation}
	Then we have that $T(\bM_{R_1}) \subseteq \bM_{R_1}$ ($T$ maps the ball above into itself). 
	
	We can also assert that $T$ is Lipschitz continuous. By definition, it suffices to verify the Lipschitz continuity of both $S$ (actually, we only require the component $S_I$) and $\tilde{S}$. 
	
	\begin{lemma}\label{lem:S-lip}
		Assume that \eqref{eq:assumption-data} holds. Then there exists $L_S>0$ such that
		\[\| S_I(\zeta_1)-S_I(\zeta_2)\|_{\bH} \leq L_S \|\zeta_1 - \zeta_2\|_{Z} \qquad \text{ for all $\zeta_1,\zeta_2\in Z$}. \]
	\end{lemma}
	\begin{proof}
		For given $\zeta_1,\zeta_2\in Z$, let $(\bu_1,p_1,\lambda_1), (\bu_2,p_2,\lambda_2) \in \bH\times Q \times W$ be the unique solutions to the decoupled Navier--Stokes equations \eqref{eq:ns1}-\eqref{eq:ns2}, with $S_I(\zeta_1)= \bu_1$, $S_I(\zeta_2) = \bu_2$, respectively. Precisely from \eqref{eq:ns2}, and using the linearity of $g$, we obtain 
		\begin{equation}
			\label{eq:aux02}
			b(\bu_1-\bu_2,(q,\xi))  = \langle \xi, g(\zeta_1-\zeta_2) \rangle_{\Gm} \quad \forall (q,\xi)\in Q\times W.\end{equation} 
		Then, for a given $(q,\xi)\in Q\times W$ and with $\xi \neq 0$, we arrive at 
		\begin{linenomath}
			\begin{align*}
				\|\bu_1-\bu_2\|_{1,\Omega}\|\xi\|_{-1/2,\Gm} & \leq   \|\bu_1-\bu_2\|_{1,\Omega}(\|q\|_{0,\Omega} + \|\xi\|_{-1/2,\Gm})\\
				& \lesssim b(\bu_1-\bu_2,(q,\xi)) 
				 = \langle \xi, g(\zeta_1-\zeta_2) \rangle_{\Gm} \\
				& \leq L_S \|\xi\|_{-1/2,\Gm} \|\zeta_1-\zeta_2\|_{1,\Omega},
			\end{align*}
		\end{linenomath}
		where we have used the inf-sup condition from Lemma~\ref{lem:B-inf}, the relation \eqref{eq:aux02}, and the Cauchy--Schwarz inequality. Then the result follows after dividing by $\|\xi\|_{-1/2,\Gm}$ on both sides of the inequality. The Lipschitz constant $L_S$ depends on the slope of the function $g$ and on the inf-sup constant for $b(\cdot,\cdot)$. 
	\end{proof}

	\begin{lemma}\label{lem:S-tild-lip}
		Assume that \eqref{eq:assumption-data} holds. Then there exists $L_{\tilde{S}}>0$ such that 
		\[\| \tilde{S}(\bw_1)-\tilde{S}(\bw_2)\|_{Z} \leq L_{\tilde{S}} \|\bw_1 - \bw_2\|_{\bH} \qquad 
		\text{for all $\bw_1,\bw_2\in \bX^g$}.\] 
	\end{lemma}
	\begin{proof}
		Consider $\bw_1,\bw_2\in \bX^g$ and the unique solutions $\theta_1,\theta_2\in Z$ of \eqref{eq:diff} associated with $\bw_1$ and $\bw_2$, respectively. Since 
		$\bw_2\in \bX^g$, we have that $\widetilde{c}(\bw_2; \tau, \tau) \geq 0$ (see \eqref{eq:ct-pos}). Let us now 
		subtract the resulting problems defined by $\tilde{S}$. This gives  
		\[c(\theta_1-\theta_2,\tau)  + \widetilde{c}(\bw_1; \theta_1, \tau) - \widetilde{c}(\bw_2; \theta_2, \tau) = 0 \quad \forall \tau \in Z_0. \]
		Then, adding and subtracting $\widetilde{c}(\bw_2; \theta_1, \tau)$ and taking $\tau=\theta_1-\theta_2$ (which is in $Z_0$ since both $\theta_1,\theta_2$ are in $Z$), we obtain 
		\begin{linenomath}
			\begin{align*}
				\|\tilde{S}(\bw_1)-\tilde{S}(\bw_2)\|^2_{1,\Omega}& = \|\theta_1-\theta_2\|^2_{1,\Omega} \lesssim c(\theta_1-\theta_2, \theta_1-\theta_2)\\
				& =  - \widetilde{c}(\bw_1-\bw_2; \theta_1, \theta_1 - \theta_2) -\widetilde{c}(\bw_2; \theta_1-\theta_2, \theta_1-\theta_2)\\
				&\leq |  \widetilde{c}(\bw_1-\bw_2; \theta_1, \theta_1-\theta_2)| - \widetilde{c}(\bw_2; \theta_1-\theta_2, \theta_1-\theta_2) \\
				& \leq \|\bw_1-\bw_2\|_{1,\Omega}\|\theta_1\|_{1,\Omega}\|\theta_1-\theta_2\|_{1,\Omega}\\
				& \lesssim  \|\bw_1-\bw_2\|_{1,\Omega}\|\theta_1-\theta_2\|_{1,\Omega},
			\end{align*}
		\end{linenomath}
		where we have used \eqref{eq:coer-c}, then \eqref{eq:continuity-widetilde_c}, and in the last step we invoked \eqref{eq:bound-} applied to the unique solution $\theta_1$ of \eqref{eq:diff}, together with  assumption \eqref{eq:assumption-data}. The Lipschitz constant $L_{\tilde{S}}$ depends on the Sobolev embedding  constant and on $R_1$. 
	\end{proof}
	
	In summary, from Lemmas \ref{lem:S-lip}-\ref{lem:S-tild-lip}, we can assert that, for $\bw_1,\bw_2,\bu_1,\bu_2\in \bH$ such that $\bu_1 = T(\bw_1)$ and $\bu_2 = T(\bw_2)$, there holds  
	\[ \|T(\bw_1) - T(\bw_2)\|_{1,\Omega} = \| S_I(\tilde{S}(\bw_1)) - S_I(\tilde{S}((\bw_2)) \|_{1,\Omega} \leq L_T \|\bw_1-\bw_2\|_{1,\Omega}, \]
	where $L_T  = \max\{ L_S,  L_{\tilde{S}}\} >0$. Finally, assuming sufficiently small data such that $L_T <1$,  the Banach fixed-point theorem gives the existence and uniqueness of solution to \eqref{eq:fixed-point} and, equivalently, to \eqref{eq:cross-flow-fv1}.

	%%%%%%%%%%%%%%%%%%%%%%%%%%%%%%%%%%%%
	\section{Finite element formulation}\label{sec:fe}
	In this section we propose a divergence-free FEM  to  approximate  problem \eqref{eq:cross-flow-fv1}. This character is required since we have used the flow incompressibility to write \eqref{eq:rewrite-convection-diffusion}. In the following we discuss all properties and stability of the method. Let us consider a shape-regular family of partitions of $\O$, denoted by $\CT_h$. While the continuous analysis has been conducted for the 2D case, we stress that the construction of the numerical method also holds in 3D. We assume that the approximations $\O_h$ of the domain $\O$ is partitioned in simplices such that for $n=2$ we have triangles, whereas tetrahedrons are considered if $n=3$.  We denote by $\Gamma_{m,h}$ to the approximation of the membrane boundary $\Gm$. Let $h_K$ be the diameter of the element $K\in\CT_h$, and let us define $h:=\max\{h_K\,:\, K\in \CT_h\}$. For the sake of uniformity, in the following we shall use the same notation to denote the FE spaces irrespective of the specific scheme.

	\subsection{Divergence-conforming approximation}\label{sec:div-conforming}
 
	For each $K$, we denote by $\bn_K$ a the unit normal vector on its boundary, which we denote by $\partial K$. We define $\CE_h:=\CE_I\cup\CE_\partial$ as the set of all facets in $\CT_h$, where $\CE_I$ is the set of all the interior facets, and $\CE_\partial$ corresponds to the set of all boundary facets in $\CT_h$. 	We define $\CE_{D}:=\CE_{\mathrm{in}}\cup\CE_{\mathrm{wall}}$, where $\CE_{\mathrm{in}}$ denotes the set of facets on the inlet $\Gi$, and $\CE_{\mathrm{wall}}$ the set of facets on the wall $\Gw$. The set that contains the facets along $\Gm$ is denoted by $\CE_{\mathrm{m}}$, and $\CE_{\mathrm{out}}$ denotes the set of facets along $\Go$. Then, we have $\CE_\partial=\CE_{D}\cup\CE_{\mathrm{m}}\cup\CE_{\mathrm{out}}$. Finally, the diameter of a given facet $e$ is denoted by $h_e$. 
 Let $K^+$ and $K^-$ be two adjacent elements on $\CT_h$, and $e:=\partial K^+ \cap \partial K^- \in \CE_I$. Given a piece-wise smooth vector-valued function $\bv$ and a matrix-valued function $\btau$, we denote by $\bv^{\pm}$ and $\btau^{\pm}$ the traces of $\bv$ and $\btau$ on the facet $e$ taken from the interior of $K^{\pm}$. Then, the jump and average for $\bv$ and $\btau$ on the facet $e$, respectively, are defined by
	\begin{equation*}
		%\label{eq:jump_and_average}
\jump{\bv\otimes\bn_e}:=\bv^+\otimes\bn_{e}^+ + \bv^-\otimes\bn_{e}^-,\qquad \mean{\btau}:=\frac{1}{2}\left(\btau^+ +\btau^- \right),
	\end{equation*}
	where the operator $\otimes$ denotes the vector product tensor $[\bu\otimes\bn]_{ij}=u_i\,n_j,\; 1\leq i,j\leq n$. 
	If $e\in \CE_B$, then we set $\jump{\bv\otimes\bn}=\bv\otimes\bn$ and $\mean{\btau}=\btau$, where $\bn$ is the unit outward normal vector to $\partial\Omega$.

    In the following we specify families of conforming and nonconforming schemes.
	\subsubsection{Divergence-conforming spaces}\label{subsec:Div-free-BDM}
	Given $k\geq0$, we define the FE spaces $\bH_h$, $Q_h$, $W_h$ and $Z_h$ for the velocity, pressure, Lagrange multiplier, and concentration, respectively, by
	\[
	\begin{aligned}
		\bH_h& :=\left\{\bv_h\in \H(\div,\O)\;:\; \bv_h\vert_K\in[\mathbb{P}_{k+1}(K)]^2,\,K\in\CT_h ,%\right.
		%\\&\hspace{4cm}\left. \ 
		\ (\bv_h\cdot \bn)|_{e\in\CE_{\mathrm{in}}} = \widehat{u},\ (\bv_h\cdot \bn)|_{e\in\CE_{\mathrm{wall}}} = 0 \right\},\\
        \bH_{h,0}& :=\left\{\bv_h\in \H(\div,\O)\;:\; \bv_h\vert_K\in[\mathbb{P}_{k+1}(K)]^2,\,K\in\CT_h ,{ \ (\bv_h\cdot \bn)|_{e\in\CE_{\mathrm{D}}} =  0} \right\},\\
		Q_h &:=\left\{q_h\in L^2(\Omega)\;:\; q_h\vert_K\in\mathbb{P}_k(K),\,K\in\CT_h \right\},\\
		W_h &:=\left\{\xi_h\in {\H^{-1/2}(\Gm)} %L^2(\overline{\Gm})
		\;:\; \xi_h\vert_{\overline{\Gamma_j}}\in\mathbb{P}_k(\overline{\Gamma_j}),\, j=1,\ldots,n_{\mathcal{E}_m} \right\},\\
		Z_h & :=\left\{\tau_h\in Z\cap C(\overline{\O})\;:\; \tau_h\vert_K\in\mathbb{P}_{k+1}(K),\,K\in\CT_h \right\}.
	\end{aligned}\]
	Here, $\mathbb{P}_r(\mathcal{O})$, for $r\geq0$, denotes the space of piecewise polynomials of degree less than or equal to $r$ defined on the entity $\mathcal{O}$, and $\widehat{u}\in\mathbb{P}_{r+1}(\Gi)$ is an interpolation of $\buin\cdot\bn$.
	For the discrete space of the Lagrange multiplier, we consider a triangulation of $\Gm$ given by $\{\Gamma_j\}_{j=1}^{n_{\mathcal{E}_m}}$, where $n_{\mathcal{E}_m}$ corresponds to the number of facets in $\Gm$.  The discrete velocity space is nonconforming in $\bH$, and correspond to the well-known divergence-conforming BDM elements family (denoted by $\mathbb{BDM}_{k+1}$) (see \cite{brezzi1991variational}). 
We end this section defining a lowest-order divergence-free $\bH^1(\Omega)$-nonconforming scheme. For $k=1$ we consider 
    \[
	\begin{aligned}
        \mathbf{CR}_h&:=\{\bv\in \bL^2(\Omega)\;:\; \bv\vert_K\in [\mathbb{P}_1(K)]^2, K\in\CT_h, \int_e \jump{\bv} =0, \forall e\in\CE_I\}, \\
		\bH_h& :=\left\{\bv_h\in \H(\div,\O)\;:\; \bv\in\mathbf{CR}_h  ,{ \ (\bv_h\cdot \bn)|_{e\in\CE_{\mathrm{in}}} = \widehat{u},\ (\bv_h\cdot \bn)|_{e\in\CE_{\mathrm{wall}}} = 0} \right\},\\
        \bH_{h,0}& :=\left\{\bv_h\in \H(\div,\O)\;:\; \bv\in\mathbf{CR}_h ,{ \ (\bv_h\cdot \bn)|_{e\in\CE_{\mathrm{D}}} =  0} \right\},\\
		Q_h &:=\left\{q_h\in L^2(\Omega)\;:\; q\vert_K\in\mathbb{P}_0(K),\,K\in\CT_h \right\},\\
		W_h &:=\left\{\xi_h\in {\H^{-1/2}(\Gm)}  
		\;:\; \xi\vert_{\overline{\Gamma_j}}\in\mathbb{P}_0(\overline{\Gamma_j}),\, j=1,\ldots,n_{\mathcal{E}_m} \right\},\\
		Z_h & :=\left\{\tau_h\in Z\cap C(\overline{\O})\;:\; \tau\vert_K\in\mathbb{P}_{1}(K),\,K\in\CT_h \right\}.
	\end{aligned}\]
Here, the space $\mathbf{CR}_h$ is nonconforming in $\bH^1(\Omega)$ and it corresponds to the lowest-order CR elements (denoted by $\mathbb{CR}_{1}$) which is coupled with piecewise constants for the pressure (see \cite{crouzeix1973conforming,girault1979finite}).

 \subsubsection{The discrete div-conforming problem}\label{subsec:BDM-RT-model-and-properties}
 As the discrete velocity now lives in $\H(\div,\O)$ and its normal trace is in $\H^{-1/2}(\partial\Omega)$, the pairings on $\Gm$ from \eqref{pairing-1},\eqref{pairing-2} suggest a discrete Lagrange multiplier space  conforming with $\H^{1/2}(\Gm)$ instead of $\H^{-1/2}(\Gm)$. In that case, in \eqref{pairing-2} one should use $\mathcal{R}^{-1}_{1/2}(g(\theta))$ instead of $g(\theta)$ (where $\mathcal{R}_{1/2}$ denotes the Riesz map between $\H^{-1/2}(\Gm)$ and its dual). However, we maintain  $W_h$ defined as conforming with $\H^{-1/2}(\Gm)$ as in the previous section. We bear in mind that the off-diagonal bilinear form (to be denoted $b_h(\cdot,\cdot)$ in \eqref{eq:bh} below) is therefore slightly different, needing   the Riesz representative  of the Lagrange multiplier.

	The remaining spaces $Q_h$ and $Z_h$ are conforming in $Q$ and $Z$, respectively. 
	With this choice of spaces, we introduce  discontinuous versions of the bilinear forms $a(\cdot,\cdot),b(\cdot,\cdot)$ and the trilinear form $\widetilde{a}(\cdot;\cdot,\cdot)$. For the first, we follow the symmetric interior penalty form given by 
	\begin{equation}
		\label{eq:bilinear_form_ah}
		\begin{aligned}
			a_h(\bu,\bv):=&\,\mu_0\int_{\O}\nabla_h\bu:\nabla_h\bv - \sum_{e\in\mathcal{E}_I \cup \CE_D}\mu_0\int_e\mean{\nabla_h\bu } :\jump{\bv\otimes\bn_e}\\
			&-\sum_{e\in\mathcal{E}_I \cup \CE_D}\mu_0\int_e\mean{\nabla_h\bv } :\jump{\bu\otimes\bn_e} + \sum_{e\in\mathcal{E}_I \cup \CE_D}\frac{\alpha_0}{h_e}\mu_0\int_e\jump{\bu\otimes\bn_e} :\jump{\bv\otimes\bn_e},
		\end{aligned}
	\end{equation}
	where $\alpha_0>0$ is the stabilisation parameter. The broken gradient operator $\nabla_h$ is defined by $\nabla_h\bu =\nabla (\bu\vert_K)$ for all $K\in\CT_h$.
	
	For the off-diagonal bilinear form we use the same functional form as $b(\cdot,\cdot)$ but the spaces are different due to the different pairings discussed above 
	\begin{equation}\label{eq:bh}
 b_h(\bv,(q,\xi)) : = -\int_\Omega q \nabla\cdot\bv+
{\langle \bv\cdot\bn,\mathcal{R}_{1/2}\xi\rangle}_{\Gamma_{m}}
 \qquad \forall \bv\in \bH_h,(q,\xi)\in Q_h\times W_h.\end{equation}
	
	For the convection term, we follow an upwind scheme (see for example \cite{buerger2019h})  defined by
	\begin{equation}
		\widetilde{a}_h(\bw;\bu,\bv)=\rho_0\int_\Omega (\bw\cdot\nabla_h\bu)\cdot\bv+ \frac{\rho_0}{2}\sum_{e \in \CE_I}\int_{ e} (\bw\cdot\bn_e - |\bw\cdot\bn_e|)(\bu^+ - \bu)\cdot\bv, \label{eq:upwind}
	\end{equation}
	where $\bu^+$ is the upwind trace of $\bu$. If $\bw\in \bH_{h,0}^0$, then the following property holds:
 $$
\widetilde{a}_h(\bw;\bu,\bu) = \frac{\rho_0}{2}\sum_{e \in \CE_I}\int_{ e}  |\bw\cdot\bn_e|[\![\bu]\!]^2 \ge 0, \quad \forall \bu\in\bH_{h,0}.
 $$

	The remaining discrete bilinear forms are the same as in Section \ref{subsec:weak_formulation}. 	Then, the resulting discrete formulation consists in finding $(\bu_h,p_h,\lambda_h,\theta_h)\in \bH_h\times Q_h\times W_h\times Z_h$ such that
	\begin{equation}
		\label{eq:cross-flow-fv1-discrete-nonconforming}
		\begin{aligned}
			a_h(\bu_h,\bv_h)+ \widetilde{a}_h(\bu_h;\bu_h,\bv_h) + b_h(\bv_h,(p_h,\lambda_h)) &=F(\bv_h) ,\\
			b_h(\bu_h,(q_h,\xi_h)) &= 
   %\langle\xi_h, {\mathcal{R}^{-1}_{1/2} \circ} g(\theta_h)\rangle_{\Gm} ,\\
  {\langle\xi_h, g(\theta_h)\rangle_{\Gm}},\\
			c(\theta_h,\tau_h) + \widetilde{c}(\bu_h;\theta_h,\tau_h)&=0,
		\end{aligned}
	\end{equation}
	for all $(\bv_h,q_h,\xi_h,\tau_h)\in {\bH_{h,0}} \times Q_h\times W_h\times Z_h$, where $F(\bv_h)$ is given by
 $$
 {F(\bv_h):=-\sum_{e\in\CE_{\mathrm{in}}}\mu_0\int_e\mean{\nabla_h\bv } :\jump{\bu_{\mathrm{in}}\otimes\bn_e} + \sum_{e\in\CE_{\mathrm{in}}}\frac{\alpha_0}{h_e}\mu_0\int_e\jump{\bu_{\mathrm{in}}\otimes\bn_e} :\jump{\bv\otimes\bn_e}.}
 $$

 %We recall that for a straight membrane such as the one depicted  in Figure \ref{fig:dibujo_canal}, the lifting
%\[
%    \bH^1(\O)\rightarrow \H^{1/2}(\Gm);\quad\bv\mapsto \bv\vert_{\Gm} \cdot\bn,
%\]
%holds (in the sense of the continuity of right inverse of the trace operator). 

%Similarly,  the corresponding discrete lifting over a mesh of $\O$ and $\Gm$ also holds.
Given $\bv\in\bH$, we define the broken $\bH_h$-norm as 
	\begin{equation}
		\label{eq:broken-H1-norm}
		\Vert \bv\Vert_{1,h}^ 2:= \Vert \bv\Vert_{0,\O}^2 + \Vert \nabla_h \bv\Vert_{0,\O}^2 + \sum_{e\in\mathcal{E}_I\cup\mathcal{E}_D}h_e^{-1}\Vert\jump{\bv\otimes\bn_e}\Vert^2_{0,e}. 
	\end{equation}
 
 \begin{lemma}
  The following bounds hold true
	\begin{linenomath}
	\begin{subequations}
		\begin{align}
			|\widetilde{a}(\bw;\bu,\bv)|&\lesssim \Vert\bw\Vert_{1,h}\,\Vert\bu\Vert_{1,h}\,\Vert\bv\Vert_{1,h},&\bw,\bu \in \bH^1(\mathcal{T}_h),\bv\in \bH_h,\label{eq:discrete-widetilde_a}\\
			|\widetilde{c}(\bw;\theta,\tau)|&\lesssim \Vert\bw\Vert_{1,h}\,\Vert\tau\Vert_{1,\Omega}\, \,\Vert\theta\Vert_{1,\Omega},&\bw\in\bH^1(\mathcal{T}_h),\theta,\tau\in Z_h\label{eq:discrete-widetilde_c},
\end{align}\end{subequations}\end{linenomath}
where
$$ 
\bH^1(\mathcal{T}_h) := \{\bv\,| \,\forall K\in\mathcal{T}_h,\; \bv\in \bH^1(K)\}.  $$
\end{lemma}
\begin{proof}
Using H\"older's inequality and the embedding result discussed in \cite{buffa2009compact}  gives 
$$
    |\widetilde{a}(\bw;\bu,\bv)|\lesssim \Vert\bw\Vert_{\bL^4(\Omega)}\,\Vert\bu\Vert_{1,h}\,\Vert\bv\Vert_{\bL^4(\Omega)}\leq\Vert\bw\Vert_{1,h}\,\Vert\bu\Vert_{1,h}\,\Vert\bv\Vert_{1,h}. 
$$
Similarly, the second result follows.
\end{proof}

	%%%%%%%%%%%%%%%%%%%%%%%%%%%%%%%%
	%%%%%%%%%%%%%%%%%%%%%%%%%%%%%%%%
	\section{Well-posedness of the divergence-conforming discrete problem}\label{sec:discrete-wellp}
	In this section we discuss the uniqueness and stability of the discrete solution to \eqref{eq:cross-flow-fv1-discrete-nonconforming}. The proof of the existence of a solution to \eqref{eq:cross-flow-fv1-discrete-nonconforming} follows exactly as in the continuous case addressed in Section \ref{sec:wellp}.
 %We investigate in this section the existence and uniqueness of a discrete solution to \eqref{eq:cross-flow-fv1-discrete-nonconforming}. 
    
We begin by showing the ellipticity of the discrete bilinear forms $a_h(\cdot,\cdot)$ and $c(\cdot,\cdot)$. 
       \begin{lemma}\label{lem:B-discoer}
        There holds:
        $$
            a_h(\bv,\bv)\gtrsim \Vert\bv\Vert_{1,h}^2\quad \forall \bv\in {\bH_{h,0}}
            \quad \mbox{and}\quad c(\tau,\tau)\gtrsim \Vert\tau\Vert_{1,\Omega}^2\quad \forall \tau\in Z_h.
        $$
        \end{lemma}
        \begin{proof}
        The first bound directly follows from \cite[Prop. 10]{hansbo2002discontinuous}. Using (\ref{eq:coer-c}) gives the second estimate.
        \end{proof}
        
             \begin{lemma}\label{lem:B-discoerinfsup}
        There holds:
        $$
           \sup_{\bv\in {\bH_{h,0}},\bv\neq\boldsymbol{0}}\frac{
			b_{h,1}(\bv,q) }{\Vert\bv\Vert_{1,h}}\gtrsim \Vert q\Vert_{0,\Omega} \quad
   \mbox{and} \quad
   \sup_{\bv\in {\bH_{h,0}},\bv\neq\boldsymbol{0}}\frac{
			b_{h,2}(\bv,\xi) }{\Vert\bv\Vert_{1,h}}\gtrsim \Vert \xi\Vert_{-1/2,\Gamma_{m,h}},
        $$
        for all $q\in Q_{h}$ and for all $\xi\in W_{h}$, where
        $$
        b_{h,1}(\bv,q):= -\int_\Omega q \nabla\cdot\bv, \qquad b_{h,2}(\bv,\xi):= {\langle \bv\cdot\bn,\mathcal{R}_{1/2}\xi\rangle}_{\Gamma_{m}}.
        $$
        \end{lemma}
        \begin{proof}
        The first bound directly follows from \cite[Prop. 10]{hansbo2002discontinuous}. The proof of the  second inf-sup condition can be done similarly to \cite[Corollary 3.5]{kashiwabara2019penalty}. 
        %gives the second stated result.
        \end{proof}
        
         As a consequence of the above lemma, we have the following result, which proves %the satisfaction of 
         an inf-sup 
         condition by the bilinear form $b_h(\cdot,(\cdot,\cdot))$.
     	\begin{lemma}\label{lem:B-disinf}
		The following discrete inf-sup condition holds 
		\[
		\sup_{\bv\in {\bH_{h,0}},\bv\neq\boldsymbol{0}} \frac{
			b_h(\bv,(q,\xi)) }
         {\Vert\bv\Vert_{1,h}}
        \gtrsim 
        \Vert q\Vert_{0,\Omega}
        +\Vert \xi \Vert_{-1/2,{\Gamma_{m,h}}} 
        \quad \forall (q,\xi)\in Q_{h}\times W_{h}. %\label{eq:inf-sup-cont}
		\]
	\end{lemma}
 where 
    \[\Vert \xi\Vert_{-1/2,\Gamma_{m,h}}= (\sum_{e\in\Gamma_{m,h}}h_e\Vert \xi\Vert_{0,e}^2)^{1/2}.\]
	\begin{proof}
    Combining the discrete inf-sup conditions discussed in Lemma \ref{lem:B-discoerinfsup}    implies the stated result. 
    \end{proof}
    In the following result we prove a global inf-sup condition of the linear part of \eqref{eq:cross-flow-fv1-discrete-nonconforming} that shall be useful to ensure the uniqueness and convergence of the discrete solution.
    \begin{lemma}\label{lem:stab-div-11}
    For each $(\bu_h,p_h,\theta_h,\lambda_h)\in {\bH_{h,0}}\times Q_h\times Z_h\times W_h$, there exists $(\bv,q,\tau,\xi)\in {\bH_{h,0}}\times Q_h\times Z_h\times W_h$ with 
    \[\triplenorm{(\bv,q,\tau,\xi)}\lesssim \triplenorm{(\bu_h,p_h,\theta_h,\lambda_h)},\]
    such that 
    \[
     \triplenorm{(\bu_h,p_h,\theta_h,\lambda_h)}^2\lesssim B((\bu_h,p_h,\theta_h,\lambda_h),(\bv,q,\tau,\xi)),   
    \]
    where 
    \[{B(\bu_h,p_h,\theta_h,\lambda_h; \bv,q,\tau,\xi):=a_h(\bu_h,\bv)+b_h(\bv,(p_h,\lambda_h))+b_h(\bu_h,(q,\xi))+c(\theta_h,\tau),}\]
    and
    \[
       \triplenorm{(\bv,q,\tau,\xi)}^2:= \Vert\bv\Vert_{1,h}^2+\|q\|_{0,\Omega}^2+\Vert \xi\Vert_{-1/2,\Gamma_{m,h}}^2+\|\tau\|_{1,\Omega}^2.
    \]
\end{lemma}
\begin{proof}
 Combining Lemma \ref{lem:B-discoer} with Lemma \ref{lem:B-disinf} leads to the stated result. 
\end{proof}

Now we are in position to prove that the solution to \eqref{eq:cross-flow-fv1-discrete-nonconforming} is unique. This is stated in the next result.
%%% Not needed for now
% \begin{lemma}\label{lem:stab-div-12}
%     For each $(\bu_h,p_h,\theta_h,\lambda_h)\in {\bH_{h,0}}\times Q_h\times Z_h\times W_h$ \cred{ and $\bw\in \bH_{h,0}^0$ with $\Vert \bw\Vert_{1,h}\le M$,
%      for sufficiently small positive $M<1$}, there exists $(\bv,q,\tau,\xi)\in {\bH_{h,0}}\times Q_h\times Z_h\times W_h$ with 
%     \[|||(\bv,q,\tau,\xi)|||\lesssim|||(\bu_h,p_h,\theta_h,\lambda_h)|||,\]
%     such that 
%     \begin{align}
%      |||(\bu_h,p_h,\theta_h,\lambda_h)|||^2\lesssim B^{(\bw)}((\bu_h,p_h,\theta_h,\lambda_h),(\bv,q,\tau,\xi)),   
%     \end{align}
%     where 
%     \begin{align*}B^{(\bw)}(\bu_h,p_h,\theta_h,\lambda_h; \bv,q,\tau,\xi)&:=a_h(\bu_h,\bv)\cred{+\widetilde{a}(\bw;\bu_h,\bv)}+b_h(\bv,(p_h,\lambda_h))\\
%     &\qquad +b_h(\bu_h,(q,\xi))+c(\theta_h,\tau)+\widetilde{c}(\bw;\theta_h,\tau),\end{align*}
%     and
%     \[
%        |||(\bv,q,\tau,\xi)|||^2:= \Vert\bv\Vert_{1,h}^2+||q||_{0,\Omega}^2+\Vert \xi\Vert_{-1/2,\Gamma_{m,h}}^2+||\tau||_{1,\Omega}^2.
%     \]
% \end{lemma}
% \begin{proof}
%  Combining Lemma \ref{lem:B-discoer} with Lemma \ref{lem:B-disinf} leads to the stated result. 
% \end{proof}

 \begin{theorem}
 {Let $(\bu_h,p_h,\theta_h,\lambda_h)$ be a solution of (\ref{eq:cross-flow-fv1-discrete-nonconforming}).  Then, the following estimate holds:
     \[\triplenorm{(\bu_h,p_h,\lambda_h,\theta_h)}\lesssim \Vert \bu_{in}\Vert_{1/2,\Gi}+g_2.\]
 
 Moreover, if 
      $\Vert\bu_h\Vert_{1,h}\le M$, for sufficiently small positive $M<1$, then
     $(\bu_h,p_h,\theta_h,\lambda_h)$ is the unique solution} 
     of (\ref{eq:cross-flow-fv1-discrete-nonconforming}).
    \end{theorem}
	\begin{proof}
    The first part  follows from Lemma \ref{lem:stab-div-11} with the continuity bounds of the bilinear forms and the lifting arguments.
   Let $(\bu_1,p_1,\theta_1,\lambda_1)$ and $(u_2,p_2,\theta_2,\lambda_2)$ be two discrete weak solutions of \eqref{eq:cross-flow-fv1-discrete-nonconforming}. Using Lemma \ref{lem:stab-div-11}, for each $({\bu}_1-\bu_2,{p}_1-p_2,{\theta}_1-\theta_2,{\lambda}_1-\lambda_2)\in {\bH_{h,0}}\times Q_h\times Z_h\times W_h $, we find $(\bv,q,\tau,\xi)\in {\bH_{h,0}}\times Q_h\times Z_h\times W_h$ with 
      \[\triplenorm{(\bv,q,\tau,\xi)}\lesssim\triplenorm{({\bu}_1-\bu_2,{p}_1-p_2,{\theta}_1-\theta_2,{\lambda}_1-\lambda_2)},\]
        such that
        $$
        \triplenorm{({\bu}_1-\bu_2,{p}_1-p_2,{\theta}_1-\theta_2,{\lambda}_1-\lambda_2)}^2\lesssim B({\bu}_1-\bu_2,{p}_1-p_2,{\theta}_1-\theta_2,{\lambda}_1-\lambda_2; \bv,q,\tau,\xi).
        $$
%      where
%      \[{B(\bu_h,p_h,\theta_h,\lambda_h; \bv,q,\tau,\xi)=a_h(\bu_h,\bv)+b_h(\bv,(p_h,\lambda_h))+b_h(\bu_h,(q,\xi))+c(\theta_h,\tau)}.\] 
      By (\ref{eq:cross-flow-fv1-discrete-nonconforming}), we can assert the bound 
      \begin{align}\label{eq:lem-div-11}
       \triplenorm{({\bu}_1-\bu_2,{p}_1-p_2,{\theta}_1-\theta_2,{\lambda}_1-\lambda_2)}^2&\lesssim B({\bu}_1,{p}_1,{\theta}_1,{\lambda}_1; \bv,q,\tau,\xi)-B(\bu_2,p_2,\theta_2,\lambda_2; \bv,q,\tau,\xi)\nonumber\\
       &\lesssim|\widetilde{a}_h(\bu_1,\bu_1,\bv)-\widetilde{a}_h(\bu_2,\bu_2,\bv)|+|\widetilde{c}_h(\bu_1,\theta_1,\tau)-\widetilde{c}_h(\bu_2,\theta_2,\tau)|\nonumber\\
       &\quad+|\langle \xi, g(\theta_1)-g(\theta_2)\rangle_{\Gamma_m}|.
      \end{align}
      Using the continuity bounds implies that
      \begin{subequations}\label{eq:lem-div-12}
      \begin{align}
      |\widetilde{a}_h(\bu_1,\bu_1,\bv)-\widetilde{a}_h(\bu_2,\bu_2,\bv)|&\le M \Vert(\bu_1-{\bu}_2)\Vert_{1,h}\;\Vert\bv\Vert_{1,h},\\
      |\widetilde{c}(\bu,\theta,\tau)-\widetilde{c}(\bu,\theta,\tau)|&\le M \big(\|\theta_1-{\theta}_2\|_{1,\Omega}+\Vert(\bu_1-{\bu}_2)\Vert_{1,h}\big)\|\tau\|_{1,\Omega},\\
      |\langle \xi, g(\theta_1)-g(\theta_2)\rangle_{\Gamma_m}|&\le L'\|\xi\|_{-1/2,\Gamma_{m,h}}\|(\theta_1-{\theta}_2)\|_{1,\Omega},% \cred{[\text{Who is L??}]},
      \end{align}          
      \end{subequations}
      where $M,L'>0$ are sufficiently small. 
       Combining (\ref{eq:lem-div-11}) and (\ref{eq:lem-div-12}) implies that
      $$
       \triplenorm{({\bu}_1-\bu_2,{p}_1-p_2,{\theta}_1-\theta_2,{\lambda}_1-\lambda_2)}^2 \lesssim 0.
      $$
      This completes the proof of the second part. 
	\end{proof}
  
	\section{Convergence of the divergence-conforming discretisation}\label{sec:cv}
	Now we turn to the derivation of a priori error bounds for the finite  element formulation proposed in  Section~\ref{sec:div-conforming}.  

    \begin{theorem}
     Let $(\bu,p,\theta,\lambda)$ and $(u_h,p_h,\theta_h, \lambda_h)$ be the continuous and discrete weak solutions of \eqref{eq:cross-flow-fv1} and \eqref{eq:cross-flow-fv1-discrete-nonconforming}, respectively. If 
     \[\Vert\bu\Vert_{1,h}\le M,\quad\mbox{and}\quad \Vert\bu_h\Vert_{1,h}\le M,\]
     for sufficiently small positive $M<1$, then
     \[\triplenorm{(\bu-\bu_h,p-p_h,\theta-\theta_h,\lambda-\lambda_h)}^2\lesssim \triplenorm{(\bu-\tilde{\bu},p-\tilde{p},\theta-\tilde{\theta},\lambda-\tilde{\lambda})}^2+\sum_{K\in\CT_h}h_K^2|\bu-\tilde{\bu}|_{2,K}^2.\]
     Moreover, if $(\bu,p,\theta,\lambda)\in \bH^{k+2}(\Omega)\cap \bH\times \H^{k+1}(\Omega)\cap Q\times \H^{k+2}(\Omega)\cap W \times \H^{k+1/2}(\Gamma_m)\cap Z$, then 
     $$
     \triplenorm{(\bu-\bu_h,p-p_h,\theta-\theta_h,\lambda-\lambda_h)}^2\lesssim h^{k+1}.    
     $$
    \end{theorem}
	\begin{proof}
	    To prove the above stated result, we first split the 
        the error into two parts as
        \begin{align}\label{eq:error-div-13}
         \triplenorm{(\bu-\bu_h,p-p_h,\theta-\theta_h,\lambda-\lambda_h)}
         &\le \triplenorm{(\bu-\tilde{\bu},p-\tilde{p},\theta-\tilde{\theta},\lambda-\tilde{\lambda})}
         \nonumber\\
&\qquad          +\triplenorm{(\tilde{\bu}-\bu_h,\tilde{p}-p_h,\tilde{\theta}-\theta_h,\tilde{\lambda}-\lambda_h)}.   
        \end{align}
        Next we derive the bound of 
      %  \begin{align}
      $  \triplenorm{(\tilde{\bu}-\bu_h,\tilde{p}-p_h,\tilde{\theta}-\theta_h,\tilde{\lambda}-\lambda_h)}$. Using Lemma \ref{lem:stab-div-11}, for each $(\tilde{\bu}-\bu_h,\tilde{p}-p_h,\tilde{\theta}-\theta_h,\tilde{\lambda}-\lambda_h)\in {\bH_{h,0}}\times Q_h\times Z_h\times W_h$, we find $(\bv,q,\tau,\xi)\in {\bH_{h,0}}\times Q_h\times Z_h\times W_h$ with  
      \[\triplenorm{(\bv,q,\tau,\xi)}\lesssim\triplenorm{(\tilde{\bu}-\bu_h,\tilde{p}-p_h,\tilde{\theta}-\theta_h,\tilde{\lambda}-\lambda_h)},\]
        such that
        $$
        \triplenorm{(\tilde{\bu}-\bu_h,\tilde{p}-p_h,\tilde{\theta}-\theta_h,\tilde{\lambda}-\lambda_h)}^2\lesssim B(\tilde{\bu}-\bu_h,\tilde{p}-p_h,\tilde{\theta}-\theta_h,\tilde{\lambda}-\lambda_h; \bv,q,\tau,\xi).
        $$
      By (\ref{eq:cross-flow-fv1-discrete-nonconforming}), it follows that 
      \begin{align}\label{eq:error-div-11}
       \triplenorm{(\tilde{\bu}-\bu_h,\tilde{p}-p_h,\tilde{\theta}-\theta_h,\tilde{\lambda}-\lambda_h)}^2&\lesssim B(\tilde{\bu},\tilde{p},\tilde{\theta},\tilde{\lambda}; \bv,q,\tau,\xi)-B(\bu_h,p_h,\theta_h,\lambda_h; \bv,q,\tau,\xi)\nonumber\\
       &\lesssim B(\tilde{\bu},\tilde{p},\tilde{\theta},\tilde{\lambda}; \bv,q,\tau,\xi)-B(\bu,p,\theta,\lambda; \bv,q,\tau,\xi)
        +\mathcal{R}_{em}\nonumber\\
       &\lesssim B(\tilde{\bu}-\bu,\tilde{p}-p,\tilde{\theta}-\theta,\tilde{\lambda}-\lambda; \bv,q,\tau,\xi)
        +\mathcal{R}_{em}, 
      \end{align}
      where
      $\mathcal{R}_{em}:=\widetilde{a}_h(\bu,\bu,\bv)-\widetilde{a}_h(\bu_h,\bu_h,\bv)+\widetilde{c}_h(\bu,\theta,\tau)-\widetilde{c}_h(\bu_h,\theta_h,\tau)+\langle \xi, g(\theta)-g(\theta_h)\rangle_{\Gamma_m}$.
      Using the continuity bounds implies that
      \begin{subequations}\label{eq:error-div-12}
      \begin{align}
&      |\widetilde{a}_h(\bu,\bu,\bv)-\widetilde{a}_h(\bu_h,\bu_h,\bv)|\le M \big(\Vert(\bu-\tilde{\bu})\Vert_{1,h}+\Vert(\bu_h-\tilde{\bu})\Vert_{1,h}\big)\Vert\bv\Vert_{1,h},\\
&      |\widetilde{c}(\bu,\theta,\tau)-\widetilde{c}(\bu,\theta,\tau)|\le M \big(\|\theta-\tilde{\theta}\|_{1,\Omega}+\|\theta_h-\tilde{\theta}\|_{1,\Omega}+\Vert(\bu-\tilde{\bu})\Vert_{1,h}
      %&\hspace{3cm}
      +\Vert(\bu_h-\tilde{\bu})\Vert_{1,h}\big)\|\tau\|_{1,\Omega},\\
   &   |\langle \xi, g(\theta)-g(\theta_h)\rangle_{\Gamma_m}|\le L'\|\xi\|_{-1/2,\Gamma_{m,h}}(\|\theta-\tilde{\theta}\|_{1,\Omega}+\|\theta_h-\tilde{\theta}\|_{1,\Omega}),
      \end{align}          
      \end{subequations}
      % L' is the Liptschitz bounding constant.
      where $M,L'>0$ are sufficiently small. 
       Combining (\ref{eq:error-div-11}) and (\ref{eq:error-div-12}) yields that
      $$
       \triplenorm{(\tilde{\bu}-\bu_h,\tilde{p}-p_h,\tilde{\theta}-\theta_h,\tilde{\lambda}-\lambda_h)}^2\lesssim
       \triplenorm{(\bu-\tilde{\bu},p-\tilde{p},\theta-\tilde{\theta},\lambda-\tilde{\lambda})}^2+\sum_{K\in\CT_h}h_K^2|\bu-\tilde{\bu}|_{2,K}^2.
      $$
      Substituting the above bound in (\ref{eq:error-div-13}) leads to the stated estimate. Using the approximation results given in \cite{john2017divergence, kashiwabara2019penalty} leads to the second stated result. 
	\end{proof}

	%%%%%%%%%%%%%%%%%%%%%%%%%%%%%%%%%%%%%%%%%%%%%%%%%%%
	%%%%%%%%%%%%%%%%%%%%%%%%%%%%%%%%%%%%%%%%%%%%%%%%%%%
	\section{Lagrange multiplier stabilisation}
	\label{sec:lagrange-multiplier}
	We now briefly present a least-squares stabilised scheme %that serve to guarantee the satisfaction of an inf-sup condition when there are 
	in the case of boundary conditions associated with permeability or slip-type.
	In the conforming case, the discrete inf-sup condition to be satisfied is given by
	\[
	\sup_{\bv_h\in \bH_h,\bv_h\neq\boldsymbol{0}}\frac{
		b_h(\bv_h,(q_h,\xi_h)) 
		%b(\bv,(q,\xi))
	}{\Vert\bv_h\Vert_{1,\Omega}}\gtrsim \Vert q_h\Vert_{0,\Omega}+\Vert \xi_h\Vert_{-1/2,\Gm} \quad \forall (q_h,\xi_h)\in Q_h\times W_h. %\label{eq:inf-sup-cont}
	\]
	However, in \cite{verfurth86} it is shown that despite choosing stable inf-sup elements together with a typical choice for the Lagrange multiplier space as above, this condition may not be satisfied.  To circumvent this difficulty, one can either enrich the velocity space with bubbles having compact support along $\Gm$ (see \cite{verfurth86} for details), or add suitable residual stabilisation in the discrete problem (see, for example \cite{verfurth91,urquiza2014weak}). We adopt the latter option. 
	%	
	%	In the conforming scheme 
	We define generic FE spaces $\bH_h\subset\bH$, $Q_h\subset Q$, $W_h\subset W$ and $Z_h\subset Z$ for the velocity, pressure, Lagrange multiplier, and concentration, respectively.
	Following  \cite{urquiza2014weak},  we first define the following mesh-dependent bilinear form $d_h:W_h\times W_h\to \mathbb{R}$: 
	\begin{equation}
		\label{eq:normH-1/2}
		d_h(\lambda_h,\xi_h)%_{-1/2,\Gamma_{m,h}}
		:=\sum_{e\in\mathcal{E}_{\mathrm{m}}}h_e\int_{e}\lambda_h\,\xi_h\,\mathrm{d}s, \qquad \forall \lambda_h,\xi_h\in W_h.
	\end{equation}
	The resulting stabilised formulation consists in finding $(\bu_h,p_h,\lambda_h,\theta_h)\in \bH_h\times Q_h\times W_h\times Z_h$ such that
	\begin{equation}
		\label{eq:cross-flow-fv1-discrete-conforming-stabilised}
		\begin{aligned}
			a(\bu_h,\bv_h)+ \widetilde{a}(\bu_h;\bu_h,\bv_h) + b(\bv_h,(p_h,\lambda_h)) +s_1((\bu_h,p_h,\lambda_h),\bv_h) &=0 ,\\
			b(\bu_h,(q_h,\xi_h))+s_2((\bu_h,p_h,\lambda_h),(q_h,\xi_h)) &= \langle\xi_h, g(\theta_h)\rangle_{\Gm} ,\\
			c(\theta_h,\tau_h) + \widetilde{c}(\bu_h;\theta_h,\tau_h)&=0,
		\end{aligned}
	\end{equation}
	for all $(\bv_h,q_h,\xi_h,\tau_h)\in \bH_h \times Q_h\times W_h\times Z_h$, where the stabilising  bilinear forms are 
	$$
	\begin{aligned}
		s_1((\bu_h,p_h,\xi_h),\bv_h)&=-\alpha_0d_h\bigl(\xi_h+ (\bsig_h\bn)\cdot\bn, \delta \mu_0(\nabla\bv_h\bn)\cdot\bn\bigr), \\
		%,\delta\mu_0(\nabla\bv\,\bn)\cdot\bn\rangle_{-1/2,\Gamma_{m,h}},\\
		s_2((\bu_h,p_h,\xi_h),(q_h,\chi_h))&=-\alpha_0d_h\bigl(\xi_h+(\bsig_h\bn)\cdot\bn,\chi_h-\delta(q_h\mathbb{I}\bn)\cdot\bn\bigr).
	\end{aligned}
	$$
	Note that for the conforming method, 
	%(Taylor--Hood elements for $\bH_h\times Q_h$ and piecewise linear and discontinuous elements for $W_h$), we have that 
	the discrete quantities $\xi_h,\,\bsig_h\bn\cdot \bn,\, \nabla\bv_h\bn\cdot \bn,\, q_h\mathbb{I}\bn\cdot\bn$ all belong to $W_h$. 
	%	
	%	The parameters $\alpha_0>0$ and $\delta$ need to be chosen in order to circumvent the inf-sup condition. 
	Note also that, for a Navier--Stokes model with slip boundary condition, \cite{verfurth91} proved that choosing $\delta=0$ and $\alpha_0$ lower than a threshold yields a stable method. As in  \cite{hughes1989new}, for $\delta=1$  we have symmetry, however a smallness condition on $\alpha_0$ is needed for the sake of  stability. For $\delta=-1$ we have the anti-symmetric variation of the method \cite{franca1991error,barbosa1992boundary}, whose main advantage is the unconditional stability with respect to $\alpha_0$. 
	
	Following \cite{urquiza2014weak} we have that possible FE families which satisfy the inf-sup condition include Taylor--Hood ($\mathbb{P}_2-\mathbb{P}_1$), the MINI element ($\mathbb{P}_{1,b}-\mathbb{P}_1$), or Crouzeix--Raviart ($\mathbb{P}_{2,b}-\mathbb{P}_{1,\text{disc}}$) (see\cite{crouzeix1973conforming}).  We remark that the $\bH^1(\Omega)$-nonconforming CR family can be used in this scheme without  stabilisation. Indeed, it is enough to consider $\bH_h$ as in Section \ref{sec:div-conforming} together with $\alpha_0=0$. This choice  reduces \eqref{eq:cross-flow-fv1-discrete-conforming-stabilised} to the discrete counterpart of \eqref{eq:cross-flow-fv1}, i.e., find $(\bu_h,p_h,\lambda_h,\theta_h)\in \bH_h\times Q_h\times W_h\times Z_h$ such that
	\begin{equation}
		\label{eq:CR-cross-flow-fv1-discrete-conforming}
		\begin{aligned}
			a_h(\bu_h,\bv_h)+ \widetilde{a}_h(\bu_h;\bu_h,\bv_h) + b(\bv_h,(p_h,\lambda_h))  &=0 ,\\
			b(\bu_h,(q_h,\xi_h))&= \langle\xi_h, g(\theta_h)\rangle_{\Gm} ,\\
			c(\theta_h,\tau_h) + \widetilde{c}(\bu_h;\theta_h,\tau_h)&=0,
		\end{aligned}
	\end{equation}
	for all $(\bv_h,q_h,\xi_h,\tau_h)\in \bH_h \times Q_h\times W_h\times Z_h$
	where $a_h(\cdot,\cdot)$ is the same as in Section \ref{subsec:BDM-RT-model-and-properties}, and
	$$
	\widetilde{a}_h(\bu_h;\bu_h,\bv_h):= \rho_0\int_\Omega (\bw\cdot\nabla_h\bu)\cdot\bv.
	$$
	Consequently, one can follow the analysis in \cite{kashiwabara2019penalty} to prove the well-posedness and convergence of this formulation.
	%%%%%%%%%%%%%%%%%%%%%%%%%%%%%%%%
	%%%%%%%%%%%%%%%%%%%%%%%%%%%%%%%%
	\section{Numerical experiments}\label{sec:experiments}
	We perform a series of computational tests using the finite element library FEniCS \cite{AlnaesEtal2015} together with the special module FeniCS$_{ii}$ \cite{kuchta2020assembly} for the treatment of bulk-surface coupling mechanisms. 
 We perform an experimental error analysis through manufactured solutions. We monitor the errors of each individual unknown, the local convergence rate, and the number of necessary Newton--Raphson iterations to achieve convergence up to a prescribed tolerance of $10^{-7}$ on the residuals. 
	By $e(\cdot)$ we  denote the error associated with the quantity $\cdot$ in its natural norm, and denote by $h_i$ the mesh~size corresponding to a refinement level $i$. The experimental convergence order is computed as
	\[r(\cdot) = \frac{\log(e_i(\cdot))- \log(e_{i+1}(\cdot))}{\log(h_i)- \log(h_{i+1})}.\]
	To compute $\|\lambda-\lambda_h\|_{s,\Gm}$ (with $s \in \{-\frac{1}{2}, \frac{1}{2}\}$ because we  use Lagrange multipliers in these two spaces) we use the characterisation of $\H^{s}(\Gm)$
		in terms of the spectral decomposition of the Laplacian operator (see, e.g., \cite{kuchta2020assembly}). 
		For this, let $R : \H^1(\Gm) \rightarrow \H^1(\Gm)$ be the bounded linear operator defined by  
		\[ 
		(Ru,v)_{1,\Gm} = (u,v)_{0,\Gm} \quad \forall u,  v\in \H^1(\Gm).
		\]
		This operator has eigenfunctions $\{r_i\}_{i=1}^\infty$ forming a basis, associated with a non-increasing sequence of positive eigenvalues $\eta_i$. Then for any 
		$u = \sum_{i=1}^\infty c_i r_i$ there holds
		\[ 
		\|u\|_{s,\Gm}^2 = \sum_{i=1}^\infty c_i^2 \eta_i^{s}\,,
		\] 
		and so $\H^{s}(\Gm)$ is the closure of the span of   $\{r_i\}_{i=1}^\infty$ 
		in this norm. 
	During the experiments, different values for the stabilisation parameters are considered in order to capture the convergence of the method. 	

   We also examine the behaviour of the schemes presented in Sections \ref{sec:div-conforming} and \ref{sec:lagrange-multiplier}, focusing slightly more on the conservative scheme from  Section \ref{sec:div-conforming}. Furthermore, as we note in the experiments below,  CR elements are more versatile as they can be used in both stabilised and non-stabilised schemes. %The study of this good behaviour will be the focus of future work.

	\subsection{Divergence-conforming test}
	\label{test-nonconforming}
	First we study the experimental convergence with respect to smooth solutions in two an three dimensions.   We consider first $\O:=(0,1)^2$ with given data. Let us consider right-hand sides and appropriate boundary conditions such that the exact solution is given by
	$$
	\bu(x,y)=\left(
	\begin{aligned}
		&\cos(\pi x)\,\sin(\pi y)\\
		&-\cos(\pi y)\,\sin(\pi x)
	\end{aligned}
	\right), \qquad p(x,y)=\sin(x^2+y^2), \qquad \theta(x,y)=e^{-xy}.
	$$
	This solution satisfies $\nabla\cdot\bu=0$ in $\Omega$, and  the physical parameters $\mu_0,\rho_0$ and $D_0$ are set to one.		Table \ref{tables-1to4} presents the error history (errors with respect to mesh refinement and Newton iteration count) for different values of $k$ and a stabilisation parameter $\alpha_0=10$. It is noted that the optimal order of convergence $O(h^{k+1})$ is attained for velocity, pressure and concentration in  their respective  norms, and for the Lagrange multiplier  in the $\H^{-1/2}(\Gm)-$norm. This confirms the analysis in Section~\ref{sec:cv}. The error for the velocity was computed using \eqref{eq:broken-H1-norm}. 

 \begin{table}[t!]
		\setlength{\tabcolsep}{3.5pt}
		\centering 
		{\small\begin{tabular}{|rcccccccccc|}
				\hline\hline
				DoF   &   $h$  & $\mathrm{e}(\bu)$  &   $r(\bu)$   &   $\mathrm{e}(p)$  &   $r(p)$  &  $\mathrm{e}(\lambda)$  &   $r(\lambda) $ &  $\mathrm{e}(\theta)$  &  $r(\theta)$  &   it \\
				\hline 
				\multicolumn{11}{|c|}{$k=0$} \\
				\hline
				971 & 0.141 & 4.69e-01 & $\star$ & 8.97e-01 & $\star$ & 2.23e-01 & $\star$ & 3.60e-02 & $\star$ &  7 \\
  3741 & 0.071 & 2.34e-01 & 1.01 & 4.58e-01 & 0.97 & 7.08e-02 & 1.66 & 1.81e-02 & 1.00 &  7 \\
  8311 & 0.047 & 1.56e-01 & 1.00 & 3.08e-01 & 0.98 & 3.84e-02 & 1.51 & 1.21e-02 & 1.00 &  7 \\
 14681 & 0.035 & 1.17e-01 & 1.00 & 2.32e-01 & 0.99 & 2.57e-02 & 1.39 & 9.04e-03 & 1.00 &  7 \\
 22851 & 0.028 & 9.33e-02 & 1.00 & 1.86e-01 & 0.99 & 1.92e-02 & 1.31 & 7.23e-03 & 1.00 &  7 \\
 32821 & 0.024 & 7.77e-02 & 1.00 & 1.55e-01 & 0.99 & 1.53e-02 & 1.25 & 6.03e-03 & 1.00 &  7 \\
				\hline
				\multicolumn{11}{|c|}{$k=1$}\\
				\hline
				 2621 & 0.141 & 2.87e-02 & $\star$ & 9.43e-02 & $\star$ & 1.29e-02 & $\star$ & 9.17e-04 & $\star$ &  7 \\
 10241 & 0.071 & 7.04e-03 & 2.03 & 2.46e-02 & 1.94 & 1.97e-03 & 2.71 & 2.30e-04 & 1.99 &  7 \\
 22861 & 0.047 & 3.11e-03 & 2.01 & 1.11e-02 & 1.97 & 6.74e-04 & 2.65 & 1.03e-04 & 1.99 &  7 \\
 40481 & 0.035 & 1.75e-03 & 2.01 & 6.25e-03 & 1.98 & 3.22e-04 & 2.57 & 5.78e-05 & 2.00 &  7 \\
 63101 & 0.028 & 1.12e-03 & 2.01 & 4.01e-03 & 1.99 & 1.85e-04 & 2.49 & 3.70e-05 & 2.00 &  7 \\
 90721 & 0.024 & 7.75e-04 & 2.00 & 2.79e-03 & 1.99 & 1.19e-04 & 2.41 & 2.57e-05 & 2.00 &  7 \\
 \hline
				\multicolumn{11}{|c|}{$k=2$}\\
				\hline
    5071 & 0.141 & 1.24e-03 & $\star$ & 4.45e-03 & $\star$ & 6.50e-04 & $\star$ & 1.03e-05 & $\star$ &  7 \\
 19941 & 0.071 & 1.51e-04 & 3.04 & 5.75e-04 & 2.95 & 6.83e-05 & 3.25 & 1.28e-06 & 3.01 &  7 \\
 41703 & 0.049 & 4.91e-05 & 3.02 & 1.90e-04 & 2.98 & 2.11e-05 & 3.16 & 4.20e-07 & 3.00 &  7 \\
 75193 & 0.036 & 2.01e-05 & 3.01 & 7.84e-05 & 2.99 & 8.35e-06 & 3.12 & 1.73e-07 & 3.00 &  7 \\
118483 & 0.029 & 1.01e-05 & 3.01 & 3.96e-05 & 2.99 & 4.12e-06 & 3.10 & 8.72e-08 & 3.00 &  7 \\
177421 & 0.024 & 5.50e-06 & 3.01 & 2.16e-05 & 2.99 & 2.21e-06 & 3.08 & 4.75e-08 & 3.00 &  7 \\
				\hline
				\hline
		\end{tabular}}
		\smallskip
		\caption{Example \ref{test-nonconforming}. 
  %\cred{[New k=2 experiment to strengthen div-free stabilised formulation choice]} 
  Error history and Newton iteration count for a FE family of $\mathbb{BDM}_{k+1}-\mathbb{P}_{k}- \mathbb{P}_{k} - \mathbb{P}_{k+1}$, with $k=0,1,2,$ for $\bu_h, p_h,\lambda_h$ and $\theta_h$, respectively, on the unit square domain $\O=(0,1)^2$. For this case the Lagrange multiplier errors are measured in the $\H^{-1/2}(\Gm)-$norm, and $\alpha_0=10(k+2)$. }
		\label{tables-1to4}
	\end{table}	
	
    Next we consider the unit cube $\O:=(0,1)^3$. Although the analysis has been performed for two dimensions, we study the performance of the method in three dimensions, where the respective tangential components are now considered in the decomposition of the stress tensor. The right-hand side and boundary conditions are chosen such that the exact solution is given by
    $$
	\bu(x,y,z)=\left(
	\begin{aligned}
		&\sin(\pi z)\,\cos(\pi y)\\
		&-\cos(\pi x)\,\sin(\pi z)\\
        &\sin(\pi x)\,\cos(\pi y)
	\end{aligned}
	\right), \quad p(x,y,z)=\sin(x^2+y^2+z^2), \quad \hspace{-0.1cm} \theta(x,y)=e^{-xyz}.
	$$
    Here we observe that $\bu$ is solenoidal, and again we consider $D_0=\rho_0=\mu_0=1$. We choose $k\in\{0,1\}$ in order to study the convergence rates on different polynomial orders.
    Table \ref{tables-3D} present the error history, mesh sizes and number of iterations for the stabilisation parameter $\alpha=10(k+2)$. Here, an optimal  $\mathcal{O}(h^{k+1})$ convergence order is observed for $k=0, \,1$. 
    
\begin{table}[t!]
		\setlength{\tabcolsep}{3.5pt}
		\centering 
		{\small\begin{tabular}{|rcccccccccc|}
				\hline\hline
				DoF   &   $h$  & $\mathrm{e}(\bu)$  &   $r(\bu)$   &   $\mathrm{e}(p)$  &   $r(p)$  &  $\mathrm{e}(\lambda)$  &   $r(\lambda) $ &  $\mathrm{e}(\theta)$  &  $r(\theta)$  &   it \\
				\hline 
				\multicolumn{11}{|c|}{$k=0$} \\
				\hline
   222 & 1.000 & 2.32e+00 & $\star$ & 1.73e+00 & $\star$ & 2.81e+00 & $\star$ & 1.80e-01 & $\star$ &  4 \\
  4478 & 0.554 & 1.11e+00 & 1.26 & 1.04e+00 & 0.87 & 9.22e-01 & 1.89 & 8.82e-02 & 1.21 &  4 \\
 22466 & 0.273 & 5.96e-01 & 0.87 & 5.62e-01 & 0.87 & 4.83e-01 & 0.91 & 4.70e-02 & 0.89 &  4 \\
139391 & 0.137 & 3.04e-01 & 0.97 & 2.78e-01 & 1.01 & 1.88e-01 & 1.36 & 2.44e-02 & 0.95 &  4 \\
256159 & 0.112 & 2.44e-01 & 1.11 & 2.19e-01 & 1.20 & 1.47e-01 & 1.25 & 1.96e-02 & 1.10 &  4 \\
\hline
				\multicolumn{11}{|c|}{$k=1$}\\
				\hline

                      675 & 1.000 & 6.87e-01 & $\star$ & 5.61e-01 & $\star$ & 3.84e-01 & $\star$ & 2.52e-02 & $\star$ &  4 \\
 14106 & 0.554 & 1.72e-01 & 2.35 & 1.31e-01 & 2.46 & 9.73e-02 & 2.33 & 5.20e-03 & 2.67 &  4 \\
 71416 & 0.273 & 4.54e-02 & 1.88 & 3.37e-02 & 1.92 & 3.14e-02 & 1.60 & 1.51e-03 & 1.75 &  4 \\
447347 & 0.137 & 1.16e-02 & 1.97 & 8.78e-03 & 1.94 & 7.77e-03 & 2.02 & 3.98e-04 & 1.92 &  4 \\
824005 & 0.112 & 7.53e-03 & 2.17 & 5.72e-03 & 2.17 & 4.90e-03 & 2.34 & 2.58e-04 & 2.19 &  4 \\
				\hline
				\hline
		\end{tabular}}
		\smallskip
		\caption{Example \ref{test-nonconforming}. Error history and Newton iteration count for a FE family of $\mathbb{BDM}_{k+1}-\mathbb{P}_{k}- \mathbb{P}_{k} - \mathbb{P}_{k+1}$, with $k=0,1,$ for $\bu_h, p_h,\lambda_h$ and $\theta_h$, respectively, on the unit cube domain $\O=(0,1)^3$. For this case the Lagrange multiplier errors are measured in the $\H^{-1/2}(\Gm)-$norm, and $\alpha_0=10(k+2)$. }
		\label{tables-3D}
	\end{table}

Finally, we show the results obtained with the $\mathbb{CR}_1-\mathbb{P}_0$ pair for velocity-pressure instead of the $\mathbb{BDM}_1-\mathbb{P}_0$ pair. We consider the same manufactured solution as before and we test the scheme in two and three dimensions. We recall that piecewise constants are used to approximate the Lagrange multiplier. Table \ref{tables-13to16} describes the behaviour of the scheme with a stabilisation parameter $\alpha_0=20$, indicating a similar  accuracy as in the $\bH^1(\Omega)$-conforming scheme presented in Table \ref{tables-1to4}.

\begin{table}[t!]
    \setlength{\tabcolsep}{3.5pt}
    \centering 
    {\small\begin{tabular}{|rcccccccccc|}
            \hline\hline
            DoF   &   $h$  & $\mathrm{e}(\bu)$  &   $r(\bu)$   &   $\mathrm{e}(p)$  &   $r(p)$  &  $\mathrm{e}(\lambda)$  &   $r(\lambda) $ &  $\mathrm{e}(\theta)$  &  $r(\theta)$  &   it \\
            \hline 
            \multicolumn{11}{|c|}{Stabilisation using $\mathbb{CR}_1-\mathbb{P}_0,\,\alpha_0=20$, $\Omega=(0,1)^2$} \\
            \hline
 %            971 & 0.141 & 4.14e-01 & $\star$ & 2.24e+00 & $\star$ & 1.35e+00 & $\star$ & 3.64e-02 & $\star$ &  7 \\
 %  3741 & 0.071 & 1.99e-01 & 1.06 & 1.19e+00 & 0.91 & 7.24e-01 & 0.90 & 1.82e-02 & 1.00 &  7 \\
 %  8311 & 0.047 & 1.31e-01 & 1.04 & 8.05e-01 & 0.96 & 4.91e-01 & 0.96 & 1.21e-02 & 1.00 &  7 \\
 % 14681 & 0.035 & 9.73e-02 & 1.02 & 6.08e-01 & 0.97 & 3.70e-01 & 0.98 & 9.06e-03 & 1.00 &  7 \\
 % 22851 & 0.028 & 7.76e-02 & 1.01 & 4.89e-01 & 0.98 & 2.96e-01 & 0.99 & 7.24e-03 & 1.00 &  7 \\
 % 32821 & 0.024 & 6.45e-02 & 1.01 & 4.08e-01 & 0.99 & 2.47e-01 & 1.00 & 6.03e-03 & 1.00 &  7 \\
 51 & 0.707 & 1.63e+00 & $\star$ & 1.92e+00 & $\star$ & 1.97e+00 & $\star$ & 1.87e-01 & $\star$ &  6 \\
   103 & 0.471 & 1.15e+00 & 0.87 & 1.89e+00 & 0.05 & 1.28e+00 & 1.08 & 1.19e-01 & 1.11 &  6 \\
   261 & 0.283 & 7.08e-01 & 0.95 & 1.38e+00 & 0.62 & 5.60e-01 & 1.61 & 7.15e-02 & 1.00 &  7 \\
   793 & 0.157 & 3.91e-01 & 1.01 & 8.52e-01 & 0.82 & 1.90e-01 & 1.84 & 4.00e-02 & 0.99 &  7 \\
  2721 & 0.083 & 2.05e-01 & 1.02 & 4.77e-01 & 0.91 & 5.56e-02 & 1.93 & 2.12e-02 & 0.99 &  7 \\
 10033 & 0.043 & 1.05e-01 & 1.01 & 2.52e-01 & 0.96 & 1.50e-02 & 1.97 & 1.10e-02 & 1.00 &  7 \\
    \hline 
            \multicolumn{11}{|c|}{Stabilisation using $\mathbb{CR}_1-\mathbb{P}_0,\,\alpha_0=20$, $\Omega=(0,1)^3$} \\
            \hline
222 & 1.000 & 2.29e+00 & $\star$ & 2.37e+00 & $\star$ & 2.65e+00 & $\star$ & 1.85e-01 & $\star$ &  4 \\
  4478 & 0.554 & 1.02e+00 & 1.37 & 8.10e-01 & 1.82 & 5.10e-01 & 2.79 & 8.84e-02 & 1.26 &  4 \\
 22466 & 0.273 & 5.48e-01 & 0.88 & 4.45e-01 & 0.85 & 2.14e-01 & 1.23 & 4.70e-02 & 0.89 &  4 \\
139391 & 0.137 & 2.80e-01 & 0.97 & 2.26e-01 & 0.98 & 6.90e-02 & 1.63 & 2.44e-02 & 0.95 &  4 \\
256159 & 0.112 & 2.25e-01 & 1.11 & 1.80e-01 & 1.15 & 4.86e-02 & 1.78 & 1.96e-02 & 1.10 &  4 \\

            \hline 
            \hline
    \end{tabular}}
    \smallskip
    \caption{Example \ref{test-nonconforming}. 
    Error history for  $\mathbb{CR}_1- \mathbb{P}_0- \mathbb{P}_0- \mathbb{P}_1$  approximations  for $\bu_h, p_h,\lambda_h$ and $\theta_h$, respectively. The Lagrange multiplier errors are measured in the $\H^{-1/2}(\Gm)-$norm and we considered two and three dimensional domains.}
    \label{tables-13to16}
\end{table}

\subsection{Lagrange stabilisation test}\label{subsec:lagrange-stabilisation-test}
{We report on experiments performed using the $\bH$-conforming stabilised scheme with Lagrange multipliers proposed in \cite{verfurth86} and presented in Section \ref{subsec:lagrange-stabilisation-test}.

First let us consider the same 2D domain and exact solutions as in Test \ref{test-nonconforming} and study the convergence of the conforming scheme using Taylor--Hood elements together with piecewise linear or constant discontinuous elements for the Lagrange multiplier. We also consider  stabilised and non-stabilised formulations in order to test the robustness of the scheme. The numerical results portrayed in Tables~\ref{tables-5to8}--\ref{tables-9to12} clearly confirm the theoretical $O(h^{k+1})$-con\-ver\-gence in the energy norm similarly to the observed/predicted in  \cite{verfurth91,urquiza2014weak}. 
The blocks in Table \ref{tables-5to8} show the error history displaying number of degrees of freedom, individual absolute errors, rates of convergence, and Newton iteration counts for the conforming discretisation using Taylor--Hood approximation of velocity-pressure, together with piecewise discontinuous elements for the Lagrange multiplier (on a submesh conforming with the bulk mesh), and piecewise quadratic and continuous functions for the concentration. The choice of $\mathbb{P}_0$ for the Lagrange multiplier shows an experimental rate of $\mathcal{O}(h^{1.5})$ for all cases.

On the other hand, the results of using linear discontinuous elements, presented in Table \ref{tables-9to12} show a noticeable deterioration of the convergence  for the Lagrange multiplier  when the stabilisation is removed.  In turn, an optimal rate of convergence $O(h^{k+1})$ is achieved with stabilisation.

We conclude this experiment by showing the result of using the pair $\mathbb{CR}_1-\mathbb{P}_0$ and $\alpha=0$. Since $\mathbb{CR}_1$ is divergence-free, we can use it to approximate \eqref{eq:cross-flow-fv1-discrete-conforming-stabilised} without stabilisation. The results obtained for this case are shown in Table \ref{table-CRstabilised}. Here we observe that the convergence for velocity, pressure and concentration is $\mathcal{O}(h)$, while for the Lagrange multiplier we obtain $\mathcal{O}(h^2)$, similar to that of Tables \ref{tables-5to8}--\ref{tables-9to12}. It is noteworthy to observe that this scheme is more computationally efficient than  Taylor--Hood but not than, e.g., the MINI-element (see Section \ref{subsec:benchmark} below).

\begin{table}[t!]
    \setlength{\tabcolsep}{3.5pt}
    \centering 
    {\small\begin{tabular}{|rcccccccccc|}
            \hline\hline
            DoF   &   $h$  & $\mathrm{e}(\bu)$  &   $r(\bu)$   &   $\mathrm{e}(p)$  &   $r(p)$  &  $\mathrm{e}(\lambda)$  &   $r(\lambda) $ &  $\mathrm{e}(\theta)$  &  $r(\theta)$  &   it \\
            \hline 
            \multicolumn{11}{|c|}{No stabilisation} \\
            \hline
            86 & 0.707 & 6.67e-01 & $\star$ & 1.86e-01 & $\star$ & 9.15e-02 & $\star$ & 2.78e-02 & $\star$ &  7 \\
            166 & 0.471 & 3.14e-01 & 1.86 & 5.24e-02 & 3.13 & 2.73e-02 & 2.99 & 1.12e-02 & 2.25 &  7 \\
            404 & 0.283 & 1.18e-01 & 1.92 & 1.07e-02 & 3.11 & 8.38e-03 & 2.31 & 3.81e-03 & 2.11 &  7 \\
            1192 & 0.157 & 3.71e-02 & 1.96 & 2.15e-03 & 2.74 & 2.93e-03 & 1.79 & 1.15e-03 & 2.04 &  7 \\
            4016 & 0.083 & 1.05e-02 & 1.99 & 5.09e-04 & 2.26 & 1.07e-03 & 1.59 & 3.20e-04 & 2.01 &  7 \\
            14656 & 0.043 & 2.80e-03 & 1.99 & 1.32e-04 & 2.04 & 3.82e-04 & 1.55 & 8.50e-05 & 2.00 &  7 \\
            \hline 
            \multicolumn{11}{|c|}{Stabilisation with $\alpha_0=0.1,\,\delta=-1$} \\
            \hline
            86 & 0.707 & 6.91e-01 & $\star$ & 3.66e-01 & $\star$ & 3.26e-01 & $\star$ & 3.18e-02 & $\star$ &  6\\
   166 & 0.471 & 3.18e-01 & 1.91 & 1.06e-01 & 3.06 & 9.40e-02 & 3.07 & 1.19e-02 & 2.43 &  7\\
   404 & 0.283 & 1.18e-01 & 1.94 & 2.13e-02 & 3.14 & 2.03e-02 & 3.00 & 3.85e-03 & 2.20 &  7\\
  1192 & 0.157 & 3.72e-02 & 1.97 & 3.59e-03 & 3.04 & 4.26e-03 & 2.65 & 1.15e-03 & 2.05 &  7\\
  4016 & 0.083 & 1.05e-02 & 1.99 & 6.36e-04 & 2.72 & 1.07e-03 & 2.17 & 3.20e-04 & 2.01 &  7\\
 14656 & 0.043 & 2.80e-03 & 1.99 & 1.39e-04 & 2.30 & 3.29e-04 & 1.79 & 8.51e-05 & 2.00 &  7\\
            \hline 
            \multicolumn{11}{|c|}{Stabilisation with $\alpha_0=0.1,\,\delta=0$} \\
            \hline
            86 & 0.707 & 6.68e-01 & $\star$ & 2.06e-01 & $\star$ & 1.45e-01 & $\star$ & 2.76e-02 & $\star$ &  7 \\
   166 & 0.471 & 3.14e-01 & 1.86 & 5.65e-02 & 3.20 & 3.56e-02 & 3.47 & 1.12e-02 & 2.23 &  7 \\
   404 & 0.283 & 1.18e-01 & 1.92 & 1.14e-02 & 3.13 & 8.86e-03 & 2.72 & 3.81e-03 & 2.10 &  7 \\
  1192 & 0.157 & 3.71e-02 & 1.96 & 2.23e-03 & 2.78 & 2.67e-03 & 2.04 & 1.15e-03 & 2.04 &  7 \\
  4016 & 0.083 & 1.05e-02 & 1.99 & 5.21e-04 & 2.29 & 9.06e-04 & 1.70 & 3.20e-04 & 2.01 &  7 \\
 14656 & 0.043 & 2.80e-03 & 1.99 & 1.34e-04 & 2.04 & 3.13e-04 & 1.60 & 8.50e-05 & 2.00 &  7 \\
            \hline 
            \multicolumn{11}{|c|}{Stabilisation with $\alpha_0=0.1,\,\delta=1$ } \\
            \hline
            86 & 0.707 & 6.63e-01 & $\star$ & 1.66e-01 & $\star$ & 8.23e-02 & $\star$ & 2.72e-02 & $\star$ &  7 \\
   166 & 0.471 & 3.13e-01 & 1.85 & 4.55e-02 & 3.19 & 2.23e-02 & 3.22 & 1.11e-02 & 2.21 &  7 \\
   404 & 0.283 & 1.18e-01 & 1.92 & 9.64e-03 & 3.04 & 7.08e-03 & 2.25 & 3.81e-03 & 2.10 &  7 \\
  1192 & 0.157 & 3.71e-02 & 1.96 & 2.07e-03 & 2.61 & 2.44e-03 & 1.81 & 1.15e-03 & 2.04 &  7 \\
  4016 & 0.083 & 1.05e-02 & 1.98 & 5.20e-04 & 2.17 & 8.75e-04 & 1.61 & 3.20e-04 & 2.01 &  7 \\
 14656 & 0.043 & 2.80e-03 & 1.99 & 1.36e-04 & 2.02 & 3.09e-04 & 1.57 & 8.50e-05 & 2.00 &  7 \\
            \hline
            \hline
    \end{tabular}}
    \smallskip
    \caption{Example 
    \ref{subsec:lagrange-stabilisation-test}.
    %sec:test-conforming}. 
    Error history for a FE family with $\mathbb{P}^2_2-\mathbb{P}_1 - \mathbb{P}_0 - \mathbb{P}_2$ 
        %Taylor--Hood + P$_0$ + P$_2$ approximation 
        for $\bu_h, p_h,\lambda_h$ and $\theta_h$, respectively. For this case the Lagrange multiplier errors are measured in the $\H^{-1/2}(\Gm)-$norm. 
        %, with no stabilisation. 
    }
    \label{tables-5to8}
\end{table}

\begin{table}[t!]
    \setlength{\tabcolsep}{3.5pt}
    \centering 
    {\small\begin{tabular}{|rcccccccccc|}
            \hline\hline
            DoF   &   $h$  & $\mathrm{e}(\bu)$  &   $r(\bu)$   &   $\mathrm{e}(p)$  &   $r(p)$  &  $\mathrm{e}(\lambda)$  &   $r(\lambda) $ &  $\mathrm{e}(\theta)$  &  $r(\theta)$  &   it \\
            \hline 
            \multicolumn{11}{|c|}{No stabilisation} \\
            \hline
            % 169 & 0.471 & 3.88e-01 & $\star$ & 2.09e-01 & $\star$ & 1.43e+00 & $\star$ & 3.65e-02 & $\star$ &  7\\
            % 409 & 0.283 & 1.60e-01 & 1.73 & 7.65e-02 & 1.97 & 9.23e-01 & 0.86 & 1.31e-02 & 2.01 &  7\\
            % 1201 & 0.157 & 5.81e-02 & 1.72 & 2.39e-02 & 1.98 & 5.78e-01 & 0.80 & 4.03e-03 & 2.01 &  7\\
            % 4033 & 0.083 & 2.00e-02 & 1.68 & 6.73e-03 & 1.99 & 3.65e-01 & 0.72 & 1.13e-03 & 2.00 &  7\\
            % 14689 & 0.043 & 6.82e-03 & 1.62 & 1.79e-03 & 2.00 & 2.37e-01 & 0.66 & 2.99e-04 & 2.00 &  7\\
            % 55969 & 0.022 & 2.35e-03 & 1.57 & 4.60e-04 & 2.00 & 1.57e-01 & 0.60 & 7.70e-05 & 2.00 &  7\\
            % Changed because we had wrong refined meshes with respect to the others cases below.
            88 & 0.707 & 6.72e-01 & $\star$ & 2.57e-01 & $\star$ & 1.98e-01 & $\star$ & 2.71e-02 & $\star$ &  7 \\
           169 & 0.471 & 3.16e-01 & 1.86 & 7.16e-02 & 3.15 & 1.20e-01 & 1.24 & 1.11e-02 & 2.21 &  7\\
           409 & 0.283 & 1.19e-01 & 1.91 & 1.39e-02 & 3.20 & 1.30e-01 & -0.16 & 3.80e-03 & 2.09 &  7\\
           1201 & 0.157 & 3.80e-02 & 1.94 & 2.58e-03 & 2.87 & 9.67e-02 & 0.50 & 1.15e-03 & 2.04 &  7\\
           4033 & 0.083 & 1.10e-02 & 1.95 & 5.42e-04 & 2.45 & 6.38e-02 & 0.65 & 3.20e-04 & 2.01 &  7\\
           14689 & 0.043 & 3.05e-03 & 1.93 & 1.29e-04 & 2.16 & 4.13e-02 & 0.66 & 8.51e-05 & 2.00 &  7\\
            \hline 
            \multicolumn{11}{|c|}{Stabilisation with $\alpha_0=0.1,\,\delta=-1$} \\
            \hline
            88 & 0.707 & 7.06e-01 & $\star$ & 3.87e-01 & $\star$ & 2.71e-01 & $\star$ & 9.54e-02 & $\star$ &  6\\
            169 & 0.471 & 3.30e-01 & 1.88 & 1.23e-01 & 2.84 & 2.17e-01 & 0.55 & 3.71e-02 & 2.33 &  7\\
            409 & 0.283 & 1.25e-01 & 1.89 & 6.04e-02 & 1.39 & 1.06e-01 & 1.41 & 1.31e-02 & 2.04 &  7\\
            1201 & 0.157 & 3.99e-02 & 1.95 & 2.20e-02 & 1.72 & 3.69e-02 & 1.79 & 4.02e-03 & 2.01 &  7\\
            4033 & 0.083 & 1.13e-02 & 1.98 & 6.53e-03 & 1.91 & 1.09e-02 & 1.92 & 1.12e-03 & 2.00 &  7\\
            14689 & 0.043 & 3.01e-03 & 1.99 & 1.77e-03 & 1.97 & 2.94e-03 & 1.97 & 2.98e-04 & 2.00 &  7\\
            \hline 
            \multicolumn{11}{|c|}{Stabilisation with $\alpha_0=0.1,\,\delta=0$} \\
            \hline
            88 & 0.707 & 6.99e-01 & $\star$ & 3.16e-01 & $\star$ & 6.85e-01 & $\star$ & 8.45e-02 & $\star$ &  6\\
            169 & 0.471 & 3.32e-01 & 1.83 & 1.53e-01 & 1.78 & 3.66e-01 & 1.55 & 3.64e-02 & 2.08 &  7\\
            409 & 0.283 & 1.26e-01 & 1.90 & 6.70e-02 & 1.62 & 1.39e-01 & 1.89 & 1.31e-02 & 2.00 &  7\\
            1201 & 0.157 & 3.99e-02 & 1.95 & 2.29e-02 & 1.83 & 4.27e-02 & 2.01 & 4.02e-03 & 2.01 &  7\\
            4033 & 0.083 & 1.13e-02 & 1.98 & 6.63e-03 & 1.95 & 1.17e-02 & 2.03 & 1.12e-03 & 2.00 &  7\\
            14689 & 0.043 & 3.02e-03 & 1.99 & 1.78e-03 & 1.98 & 3.06e-03 & 2.03 & 2.98e-04 & 2.00 &  7\\
            \hline 
            \multicolumn{11}{|c|}{Stabilisation with $\alpha_0=0.1,\,\delta=1$ } \\
            \hline
            88 & 0.707 & 7.29e-01 & $\star$ & 5.21e-01 & $\star$ & 1.00e+00 & $\star$ & 8.25e-02 & $\star$ &  6\\
            169 & 0.471 & 3.40e-01 & 1.88 & 2.14e-01 & 2.19 & 4.70e-01 & 1.87 & 3.64e-02 & 2.02 &  7\\
            409 & 0.283 & 1.27e-01 & 1.92 & 7.78e-02 & 1.98 & 1.66e-01 & 2.04 & 1.31e-02 & 2.00 &  7\\
            1201 & 0.157 & 4.01e-02 & 1.97 & 2.41e-02 & 1.99 & 4.74e-02 & 2.13 & 4.02e-03 & 2.01 &  7\\
            4033 & 0.083 & 1.13e-02 & 1.99 & 6.75e-03 & 2.00 & 1.24e-02 & 2.11 & 1.13e-03 & 2.00 &  7\\
            14689 & 0.043 & 3.02e-03 & 1.99 & 1.79e-03 & 2.00 & 3.15e-03 & 2.07 & 2.99e-04 & 2.00 &  7\\
            \hline
            \hline
    \end{tabular}}
    \smallskip
    \caption{Example \ref{subsec:lagrange-stabilisation-test}. Error history for a FE family with $\mathbb{P}^2_2-\mathbb{P}_1 - \mathbb{P}_1 - \mathbb{P}_2$ 
        %Taylor--Hood + P$_0$ + P$_2$ approximation 
        for $\bu_h, p_h,\lambda_h$ and $\theta_h$, respectively. For this case the Lagrange multiplier errors are measured in the $\H^{-1/2}(\Gm)-$norm. }
    \label{tables-9to12}
\end{table}

\begin{table}[t!]
\setlength{\tabcolsep}{3.5pt}
\centering 
{\small\begin{tabular}{|rcccccccccc|}
        \hline\hline
        DoF   &   $h$  & $\mathrm{e}(\bu)$  &   $r(\bu)$   &   $\mathrm{e}(p)$  &   $r(p)$  &  $\mathrm{e}(\lambda)$  &   $r(\lambda) $ &  $\mathrm{e}(\theta)$  &  $r(\theta)$  &   it \\
        \hline 
        \multicolumn{11}{|c|}{No stabilisation} \\
        \hline
       51 & 0.707 & 1.62e+00 & $\star$ & 5.34e-01 & $\star$ & 7.05e-01 & $\star$ & 1.81e-01 & $\star$ &  6\\
   103 & 0.471 & 1.13e+00 & 0.88 & 3.55e-01 & 1.00 & 4.48e-01 & 1.12 & 1.22e-01 & 0.97 &  7\\
   261 & 0.283 & 7.02e-01 & 0.93 & 1.85e-01 & 1.28 & 2.08e-01 & 1.50 & 7.33e-02 & 1.00 &  7\\
   793 & 0.157 & 3.96e-01 & 0.97 & 8.55e-02 & 1.32 & 7.36e-02 & 1.77 & 4.04e-02 & 1.01 &  7\\
  2721 & 0.083 & 2.11e-01 & 0.99 & 4.02e-02 & 1.19 & 2.19e-02 & 1.90 & 2.13e-02 & 1.01 &  7\\
 10033 & 0.043 & 1.09e-01 & 1.00 & 1.98e-02 & 1.07 & 5.99e-03 & 1.96 & 1.10e-02 & 1.00 &  7\\
        \hline
        \hline
\end{tabular}}
\smallskip
\caption{Example 
\ref{subsec:lagrange-stabilisation-test}.
%sec:test-conforming}. 
Error history for a FE family with $\mathbb{CR}_1-\mathbb{P}_1 - \mathbb{P}_0 - \mathbb{P}_1$ 
    %Taylor--Hood + P$_0$ + P$_2$ approximation 
    for $\bu_h, p_h,\lambda_h$ and $\theta_h$, respectively. For this case the Lagrange multiplier errors are measured in the $\H^{-1/2}(\Gm)-$norm. In this case, we take $\delta= 1$.
    %, with no stabilisation. 
}
\label{table-CRstabilised}
\end{table}

	\subsection{Filtration with osmotic effects}
	\label{sec:simulation2d}
	Let us consider a membrane channel \textit{unit} whose length is defined by a subsection of the channel that allows a fully developed flow \cite{luo2020hybrid}. The channel length of the channel is given by  $L=1.5\cdot 10^{-2}$\,m, whereas the physical parameters are given below \cite{bernales2017prandtl,nayar2016thermophysical}:
	$$
	\begin{aligned}
		&\kappa=4955.144\,J/\mathrm{mol}, \quad A_0=1.189\cdot10^{-11} \mathrm{m}\,\mathrm{Pa}^{-1}\mathrm{s}^{-1},\quad \mu_0=8.9\cdot10^{-4}\mathrm{Kg}\,\mathrm{m}^{-1}\mathrm{s}^{-1},\\
		&\rho_0=1027.2\,\mathrm{Kg\,m}^{-3},\quad D_0=1.5\cdot10^{-9} \mathrm{m}^2\,\mathrm{s}^{-1},\quad \Delta P = 4053000\,\mathrm{Pa}.
	\end{aligned}
	$$
	With respect to the boundary conditions on the inlet, we consider the following:
	$$
	u_0= 1.29\times \cdot10^{-1}\mathrm{m\,s}^{-1}, \quad 
		\theta_{\mathrm{in}}= 600 \text{ mol\,m}^{-3}, \quad 
		 \bu\cdot\bn = g(\theta)=A_0(\Delta P - \kappa\theta) \qquad \text{on } \Gamma_{\mathrm{m}}.
	$$
	The runs in this example are done with a second-order H(div)-conforming discretisation (taking $k=1$) and we consider two scenarios: First a channel with a membrane at $\Gm$, while the wall conditions are kept at $\Gi$. The inlet velocity field is given by
	$$
	\bu\cdot\bn=6u_{0}(y+\widetilde{d})(\widetilde{d}-y)/{\widetilde{d}}^2,
	$$
	where $\widetilde{d}=d/2.$ 
	The second scenario consists of a channel with membranes, where $\Gw=\Gm$ is assumed. In this case, we study the behaviour of the salt profile at the boundary and compare the results at $\Gi$ with a Berman flow. To this end, we take the inlet condition as
	$$
	\bu\cdot\bn=\left(u_0 - v_w\frac{2x}{d}\right)\left[\frac{3}{2}(1-\lambda^2) \right]\left[1 -\frac{\mathrm{Re}}{420}\left( 2 - 7\lambda^2 - 7\lambda^4\right) \right],
	$$
	where $Re=\frac{v_w (d/2)}{\mu_0/\rho_0}$,  $v_w=A_0(\Delta P - \kappa \theta_{\mathrm{in}})$,  $\lambda=2y/d$.
 \begin{figure}[t!]
 \begin{center}
     \raisebox{-4mm}{$\bu_{2,h}$} \\
     \includegraphics[width=.96\textwidth]{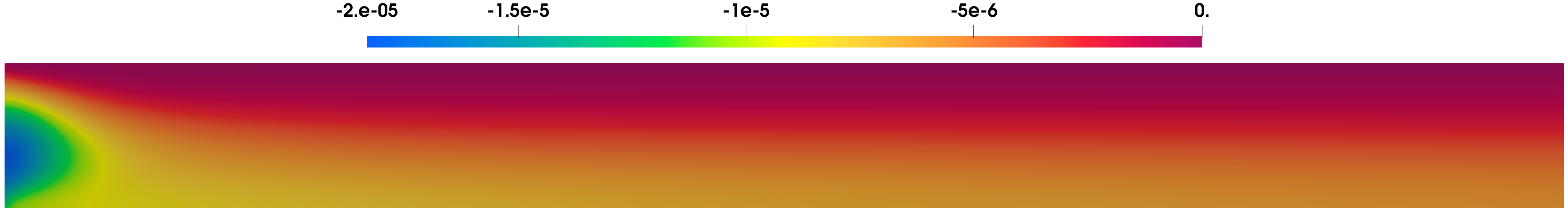}\\
     \includegraphics[width=.96\textwidth]{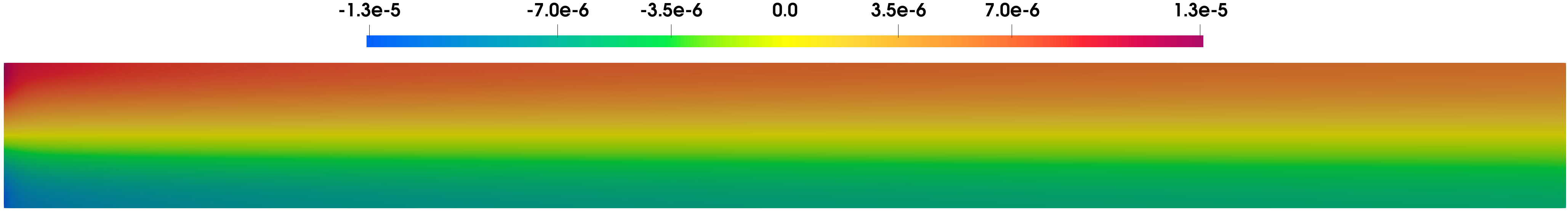}\\[2ex]
      \raisebox{2mm}{$\theta_h$} \quad \includegraphics[width=.5\textwidth]{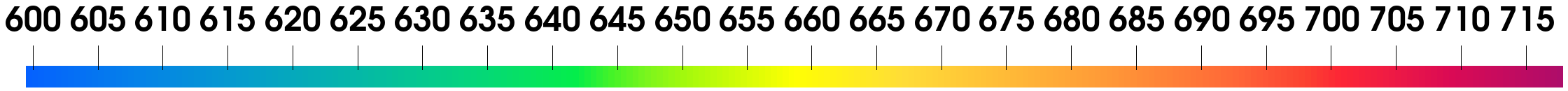}\\
      \includegraphics[width=.96\textwidth]{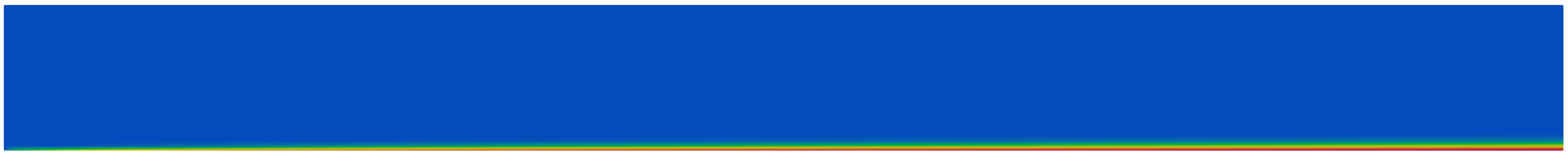}\\
      \includegraphics[width=.96\textwidth]{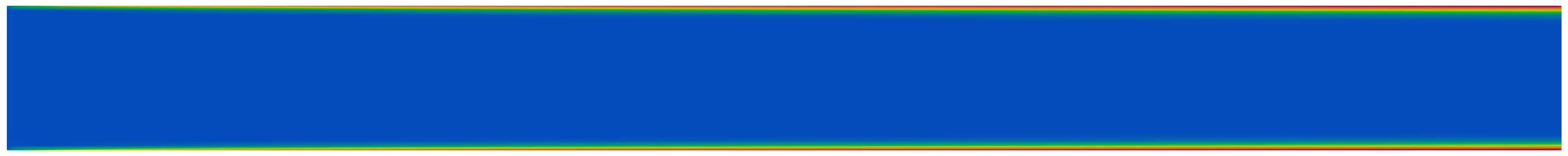}
 \end{center}
 \caption{Example \ref{sec:simulation2d}. Scenario 1 (first and third panels) and Scenario 2 (second and bottom panels).  Scaled representation of the computed velocity component $\bu_{h,1}$ and concentration profile in a channel with membrane at $\Gm$ (scenario 1) and $\Gw=\Gm$ (scenario 2).}
 \label{fig:channel-nospacer-velprescon-and-wall}
	\end{figure}

	To capture the velocity behaviour at the inlet, as well as the maximum permeate velocity, a free tangential stress $(\bsig\bn)\cdot\bt=0$ is imposed at $\Gi$. Similar results for the dual membrane channel are obtained if we consider $\bu_2$ as the corresponding velocity component on a Berman flow.

	The results for the first and second scenarios are depicted in Figure \ref{fig:channel-nospacer-velprescon-and-wall}. For the first case we can see that the velocity  near the membrane is affected by the porosity and the transmembrane pressure. Due to the minimal amount of salt compared to the rest of the membrane, at the inlet we observe a high flux in the normal direction to the membrane.  
 In Figure \ref{fig:channel-comparacion-vel-con-pres} we compare the performance of both channels. The concentration profile and permeate velocity are highly dominated by the transmembrane pressure, irrespective of the choice of inlet profiles. On the other hand, we observe that the concentration profile at the membrane increases as we approach the end of the channel, consequently decreasing the permeate velocity. This is accompanied by a linear pressure drop, which behaves similarly for both scenarios.	
	\begin{figure}[t!]
		\centering
		\includegraphics[width=0.32\textwidth]{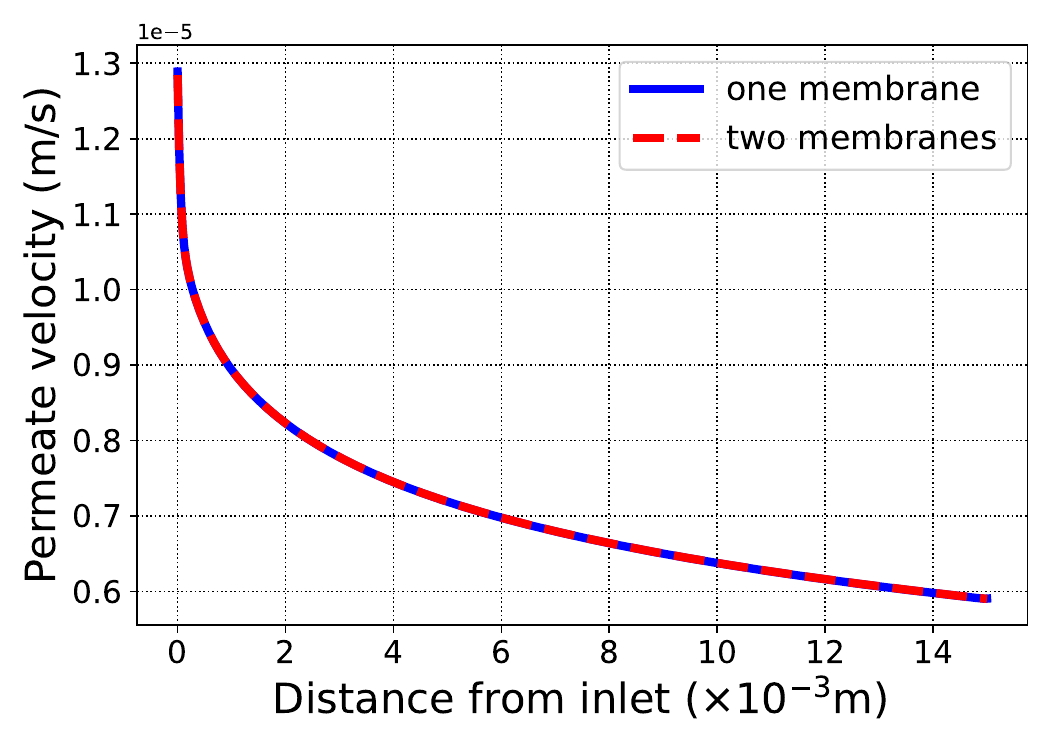}
		\includegraphics[width=0.32\textwidth]{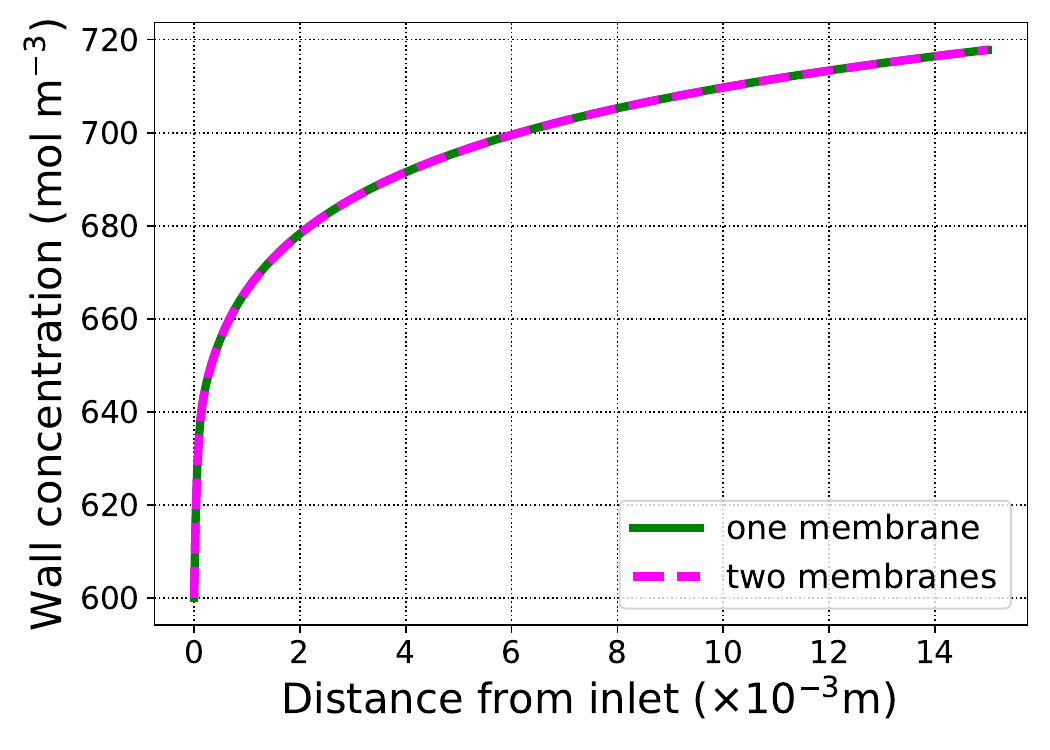}        \includegraphics[width=0.32\textwidth]{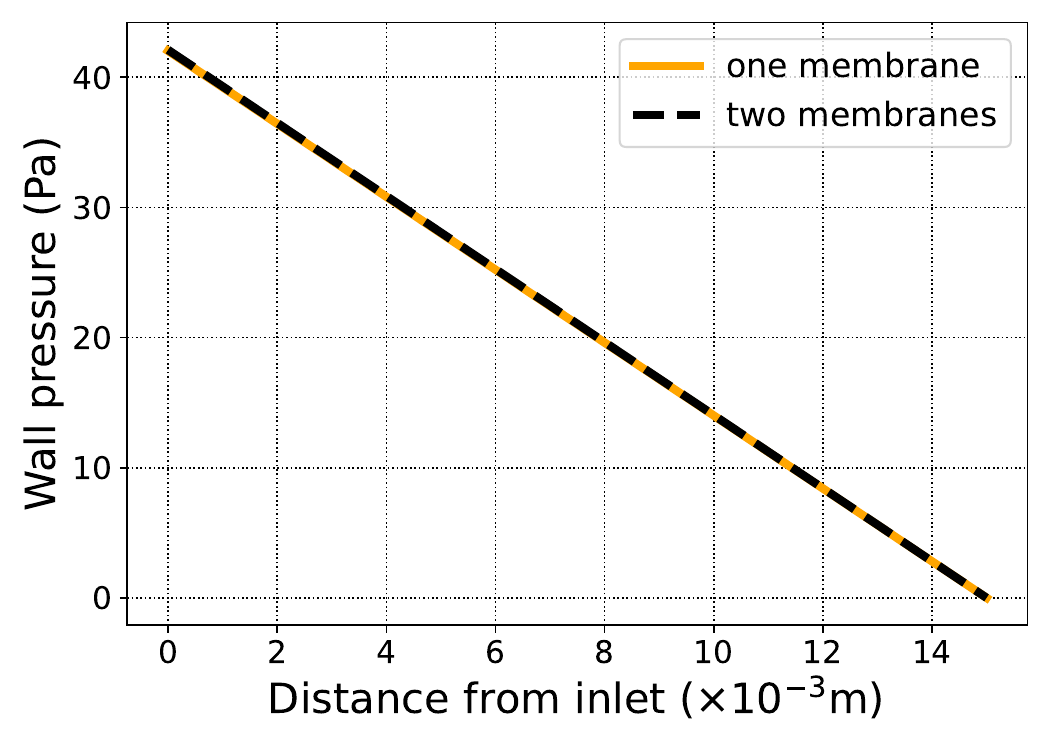}
		\caption{Example \ref{sec:simulation2d}. Comparison along $\Gm$ between permeate velocities, concentration profiles and pressures between the two channel scenarios. }
		\label{fig:channel-comparacion-vel-con-pres}
	\end{figure}

	\subsection{A channel with a spacer}
    \label{test-spacer}
We now study the behaviour of the model when an obstacle, serving a a spacer located in the middle of the channel, is considered.  The spacer corresponds to the cross section of a cylinder, i.e., a circle, with diameter $3.6\cdot10^{-4}$m and tangent to the membrane. The boundary conditions for the spacer are the same as $\Gw$. The velocities to be tested in this experiment are $u_0=5\cdot 10^{-2}$m$/$s and $u_0=1.29\cdot 10^{-1}$m$/$s, and the effect of the salt concentration boundary layer along the membrane is studied.
    
    In Figure \ref{fig:vel_and_concentration_comparison} we observe  velocity and concentration profiles when different inlet velocities are considered. The flow exhibits recirculating zones caused by the spacer,  inducing an accumulation of salt near the spacer. Moreover, near the tangent point we observe the maximum salt concentration. This is described in Figure \ref{fig:spacer_channel-comparacion-vel-con-pres}, where we observe that the high velocity profile yields a lower concentration of salt along the membrane despite the higher recirculating zone. Also, the gauge pressure drop observed for the high velocity profile is more pronounced around the spacer boundary point, as expected.
\begin{figure}[t!]
\begin{center}
 \raisebox{-2mm}{$|\bu_h|$} \\
 \includegraphics[width=0.495\textwidth]{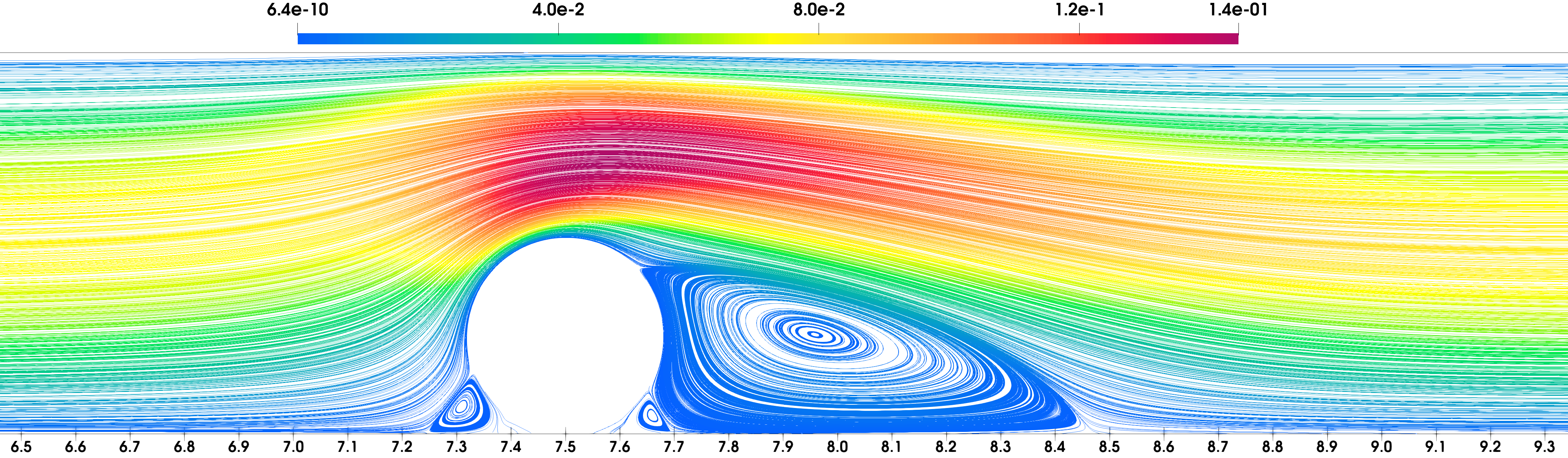}
    \includegraphics[width=0.495\textwidth]{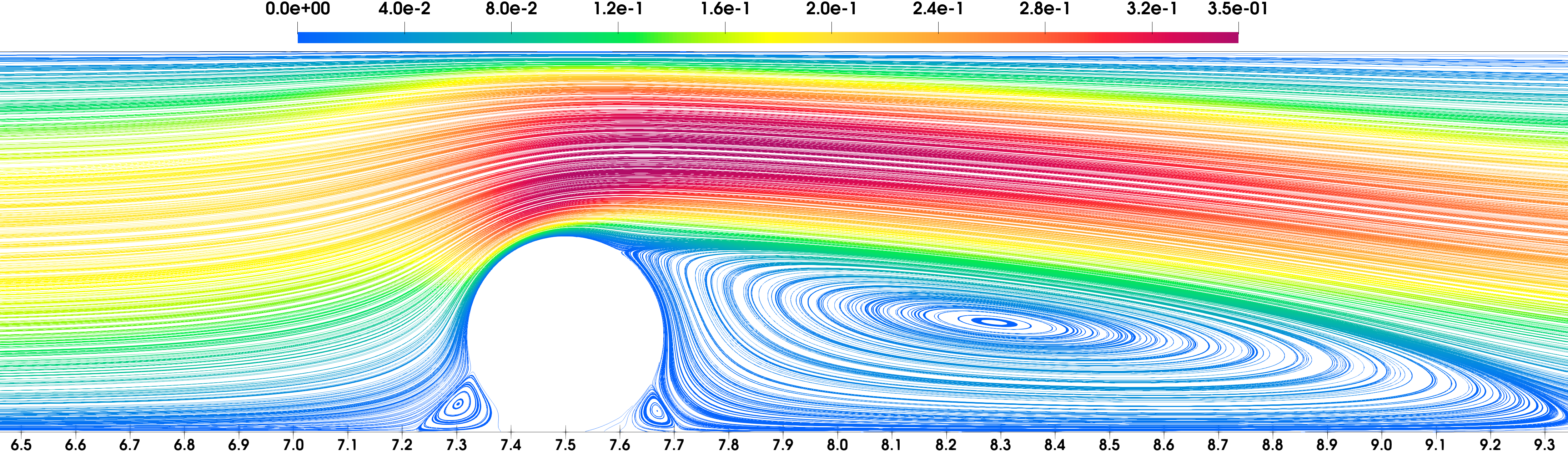} \\[1ex]
  \raisebox{1mm}{$\theta_h$} \quad  \includegraphics[width=0.6\textwidth]{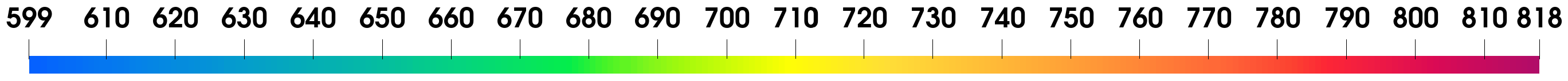}\\
    \includegraphics[width=0.495\textwidth]{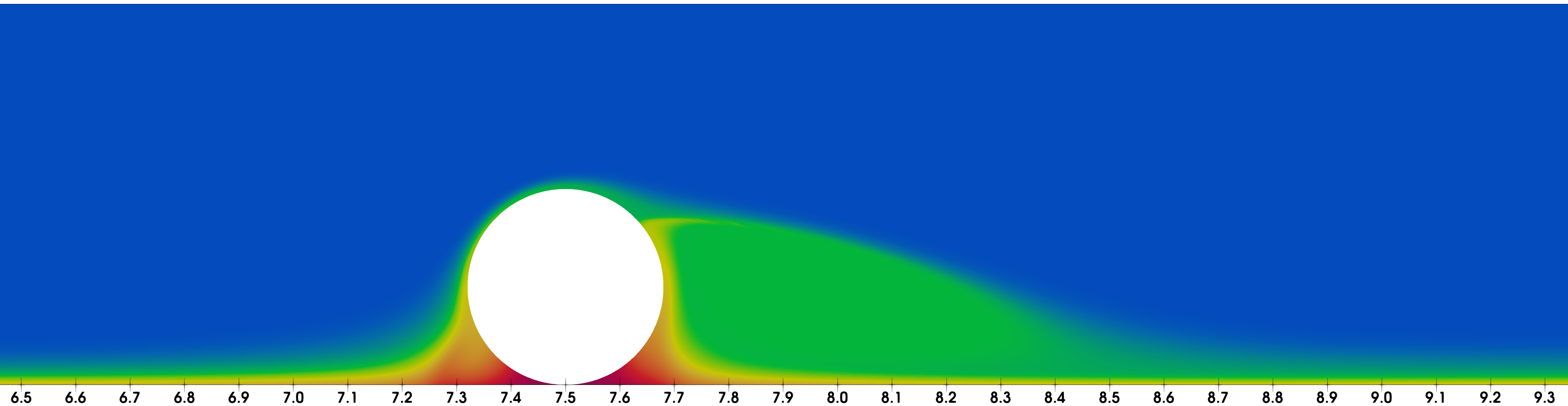}
    \includegraphics[width=0.495\textwidth]{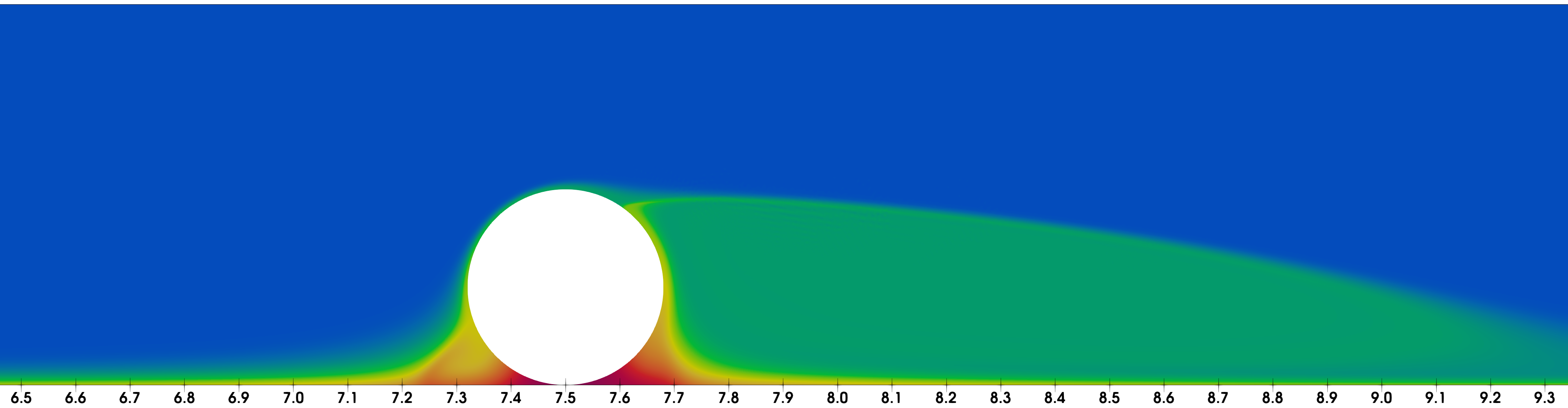}
\end{center}
\caption{Velocity streamlines (top panels) and concentration profiles (bottom panels) around the spacer in a cavity-type configuration and inlet velocities $u_0=5.0\cdot 10^{-2}$m$/$s (left) 
   and  $u_0=1.29\cdot 10^{-1}$m$/$s. The bottom numbers indicate distance from inlet (in $\times 10^{-3}$m).}
    \label{fig:vel_and_concentration_comparison}
    \end{figure}
    \begin{figure}[t!]
		\centering
		\includegraphics[width=0.32\textwidth]{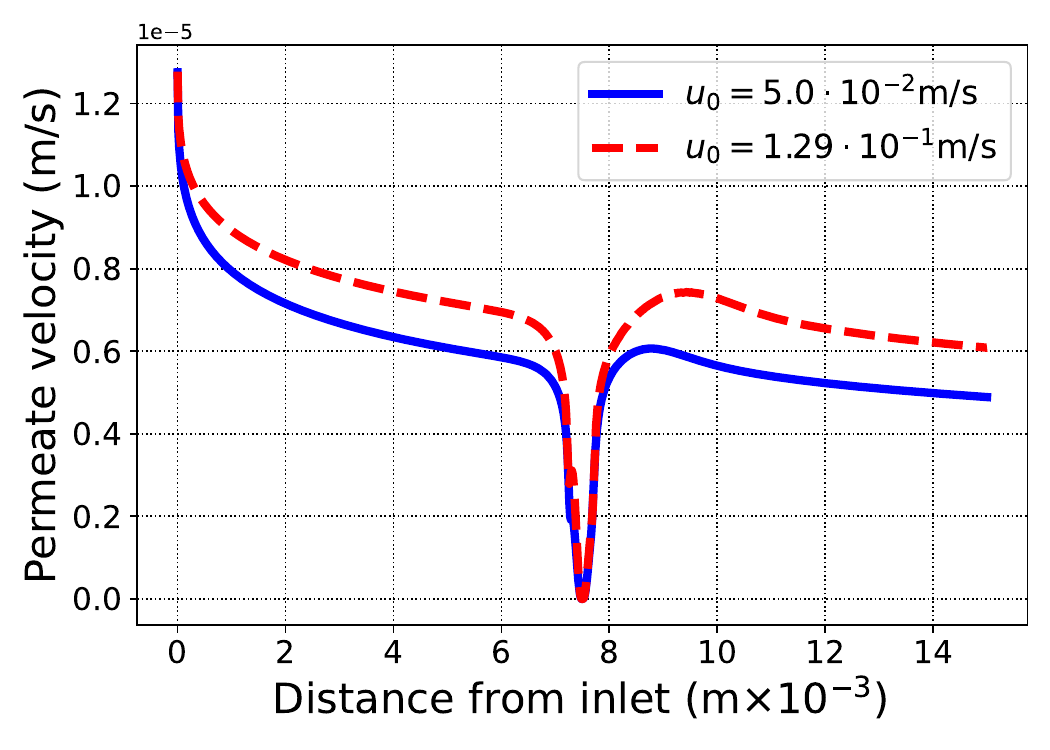}
		\includegraphics[width=0.32\textwidth]{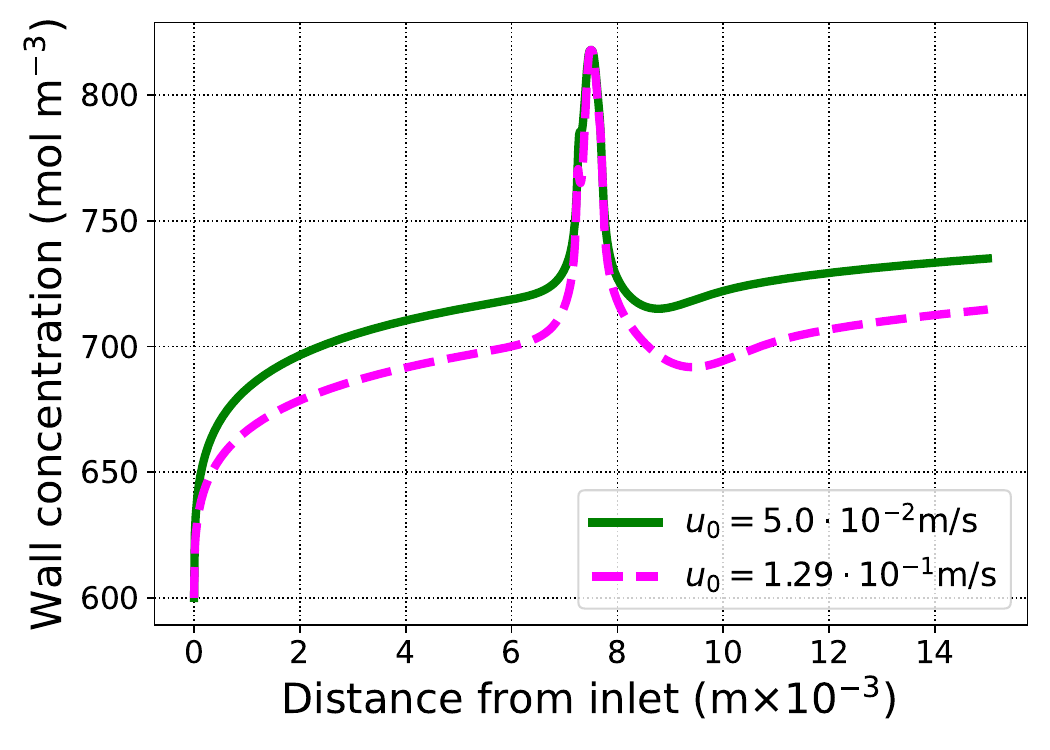}        \includegraphics[width=0.32\textwidth]{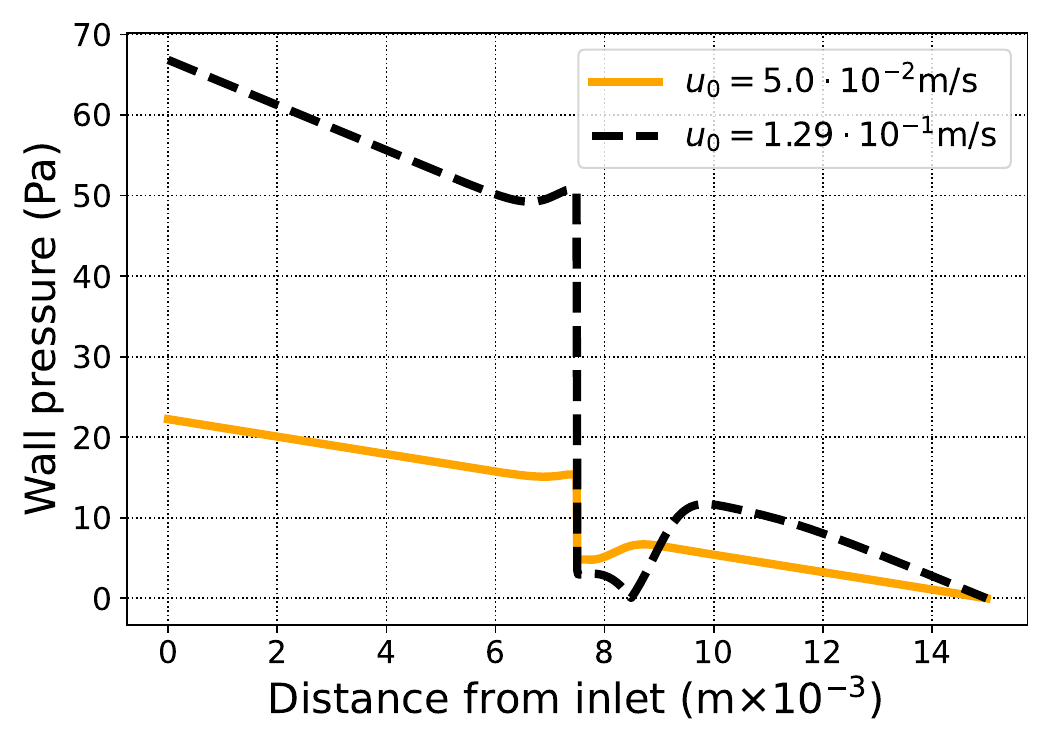}
		\caption{Example \ref{test-spacer}. Comparison along $\Gm$ between permeate velocities, concentration profiles and pressures between the two velocities scenarios in the channel with cavity-type spacer configurations. }
		\label{fig:spacer_channel-comparacion-vel-con-pres}
	\end{figure}

\subsection{Benchmark on a 3D geometry}\label{subsec:benchmark}
We conclude with a simple benchmark test to compare the effectiveness of the different numerical schemes discussed in the paper. We focus on the unit cube $\Omega:=(0,1)^3$ partitioned into  $18458$ regular tetrahedra and having $610$ facets on $\Gm$ (similar in size to the fourth mesh level on the 3D experiments in Section \ref{test-nonconforming}).

We measure the CPU time needed for matrix assembly (denoted \emph{Assembly}), the percentage of non-zero entries in the matrix (\emph{Non-zero pct}) and the CPU time to solve the system (\emph{Solve}). This information is provided by the Python library \texttt{Time} in conjunction with the FEniCS \texttt{DEBUG log}. In the case of the PETSc library, we employ the PETSc Krylov solver for handling the linear system, coupled with the direct solver MUMPS. These metrics are captured during the execution of the first Newton iteration.

For the test, we consider the div-free schemes discussed in Section \ref{sec:div-conforming} (which we denote by BDM-divfree and CR-divfree), together with the Lagrange multiplier stabilised schemes based on Taylor--Hood elements denoted by TH-stabilised-P$k$ (or TH-nonstabilised-P$k$ for the $\alpha_0=0$ variation), $k=0,1$. The $k$ represents the degree of the polynomial order for the Lagrange multiplier space. Similarly, we denote by MINI-stabilised-P$k$ (resp. MINI-nonstabilised-P$k$) the variant where the MINI-element family is used. Finally, the results coming from the first Newton iteration on the scheme \eqref{eq:CR-cross-flow-fv1-discrete-conforming} are denoted by CR-nonstabilised.

The results are portrayed in Table \ref{tabla:benchmark3D}. We observe that the divergence-free schemes have a smaller assembly time, but the extra entries (due to the facet jumps and averages) yields a solving time of 16s approximately. Perhaps surprisingly,  CR-nonstabilised has the best assembly time of all the test cases, and the solve time is half of its div-free variant. On the other hand, it is important to note the performance hit when considering the stabilisation of \eqref{eq:cross-flow-fv1-discrete-conforming-stabilised}. The solving time is much higher for the non-stabilised variations (except CR), even failing to converge in the case of MINI-nonstabilised-P1. A clear performance advantage of the MINI family is observed, which is expected because they are the cheapest inf-sup stable FE. The non-zero percentage was kept below $0.1\%$ in all cases, but the TH variants show the lowest sparsity pattern. Also,  the div-free and CR-nonstabilised schemes are the ones with the lowest amount of non-zero entries in the global matrix. The results suggest that BDM-divfree, CR-divfree and CR-nonstabilised are preferred if $\nabla\cdot\bu_h=0$ is required, while MINI schemes should be used otherwise. We also tested on the Newton  iterations required to solve \eqref{eq:CR-cross-flow-fv1-discrete-conforming} as a function of $\delta$, where we observed no significant difference. Here we took $\alpha_0=0.1$, similarly to Section \ref{subsec:lagrange-stabilisation-test}.

\begin{table}[!h]\centering
	{
		\caption{Example \ref{subsec:benchmark}. Benchmark of the different methods for solving the coupled flow-transport problem in $\O:=(0,1)^3$ with $18458$ tetrahedra and $610$ facets on $\Gm$. The symbol -- represents a failed Newton iteration.}
\begin{tabular}{lrccc}
\hline
\hline
	Scheme & DoF & \emph{Assembly} (s) & \emph{Solve} (s) & \emph{Non-zero pct} (\%)  \\ \midrule
	BDM-divfree             & 139391          &   3.05455    &  16.6812  &     0.04734          \\
	CR-divfree              & 139391          &   2.10487    &  16.7165  &     0.04734       \\
    CR-nonstabilised        & 139391          &   0.87892    &  7.25941  &      0.01856      \\
	TH-stabilised-P0        & 118431         &   5.00461    &  9.84011  &      0.08100     \\
    TH-stabilised-P1        & 119651         &   5.02505    &  10.2191  &     0.08005      \\
	TH-nonstabilised-P0     & 118431          &   3.82103    &  18.4420  &    0.08027     \\
    TH-nonstabilised-P1     & 119651          &   3.80423    &  25.4627  &    0.07926      \\
    MINI-stabilised-P0       & 76349         &   3.18728    &  2.31272  &      0.05787   \\
    MINI-stabilised-P1       & 77569         &   3.13314    &  2.24265  &     0.05700     \\
	MINI-nonstabilised-P0   & 76349          &   2.63304    &  2.49390  &    0.05614     \\
    MINI-nonstabilised-P1   & 77569          &   2.66941    &  --  &    0.05512        \\
    \hline
    \hline
\end{tabular}}
\label{tabla:benchmark3D}
\end{table}

\section*{CRediT authorship contribution statement}
A. Khan, D. Mora, R. Ruíz-Baier, J. Vellojin: Conceptualization,
Methodology, Writing-original draft, Review $\&$ Editing.

\section*{Data availability statement}
The data that support the findings of this research are available from the corresponding author upon reasonable request.

\section*{Declaration of competing interest}
The authors declare that they have no known competing financial interests or personal relationships that could have appeared to influence the work reported in this paper.
 %%%%%%%%%%%%%%%%%%%%% REFERENCES %%%%%%%%%%%%%%%%%%%%%%%%%%%%%%%%%%
	\bibliographystyle{siam}
	\bibliography{biblio}

%\appendix

\end{document}